\documentclass[reqno]{amsart}     
\usepackage{amssymb, amsmath, amsthm, epsf, epsfig}
\usepackage[a4paper,vmargin={3cm,3cm},hmargin={3cm,3cm}]{geometry}

\usepackage{bbm}
\usepackage{xcolor}
\usepackage{enumerate}
\usepackage{comment}
\usepackage{hyperref}

\renewcommand{\(}{\left(}
\renewcommand{\)}{\right)}
\newtheorem{theo}{Theorem}
\newtheorem{prop}{Proposition}
\newtheorem{lemma}{Lemma}
\newtheorem{cor}{Corollary\!\!}

\newtheorem{ncor}{Corollary}
 
\theoremstyle{definition}

\newtheorem{df}{Definition}
\newtheorem{ex}{Example}

\theoremstyle{remark}

\newtheorem{rem}{Remark\!\!}
 
\newtheorem{nrem}{Remark}

\newcommand{\beq}{\begin{equation}} 
\newcommand{\eeq}{\end{equation}} 
\newcommand{\bal}{\begin{align}} 
\newcommand{\eal}{\end{align}} 
\newcommand{\bals}{\begin{align*}} 
\newcommand{\eals}{\end{align*}} 
\newcommand{\barr}[1]{\begin{array}{#1}} 
\newcommand{\earr}{\end{array}}

\newcommand{\bth}{\begin{theo}} 
\newcommand{\bl}{\begin{lemma}} 
\newcommand{\el}{\end{lemma}} 
\newcommand{\bp}{\begin{prop}} 
\newcommand{\ep}{\end{prop}} 
\newcommand{\bdf}{\begin{df}} 
\newcommand{\edf}{\end{df}} 
\newcommand{\brem}{\begin{rem}} 
\newcommand{\erem}{\end{rem}} 
\newcommand{\bnrem}{\begin{nrem}} 
\newcommand{\enrem}{\end{nrem}} 
\newcommand{\bex}{\begin{ex}} 
\newcommand{\eex}{\end{ex}} 
\newcommand{\bcor}{\begin{cor}} 
\newcommand{\ecor}{\end{cor}} 
\newcommand{\bncor}{\begin{ncor}} 
\newcommand{\encor}{\end{ncor}} 
\newcommand{\bpf}{\begin{proof}} 
\newcommand{\epf}{\end{proof}}

\def\({\left(} 
\def\){\right)}

\newcommand{\mM}{\mathsf{M}} 
\newcommand{\mC}{\mathsf{C}} 
\newcommand{\mB}{\mathsf{B}} 
\newcommand{\cB}{\mathcal{B}} 
\newcommand{\mO}{\mathsf{O}} 
\newcommand{\cT}{\mathcal{T}} 
\newcommand{\atv}{\,{\buildrel d \over \approx}\,}
\newcommand{\co}{\mathrm{c}}
\newcommand{\ve}{\mathrm{v}}

\newcommand{\convd}{\,{\buildrel d \over \longrightarrow}\,}
\newcommand{\convp}{\,{\buildrel p \over \longrightarrow}\,}
\renewcommand{\Pr}[1]{\mathbb{P}(#1)}

\newcommand{\Ex}[1]{\mathbb{E}[#1]}
\newcommand{\Exb}[1]{\mathbb{E}\left[ #1 \right]}
\newcommand{\Va}[1]{\mathbb{V}[#1]}

\numberwithin{equation}{section}

\title{Cut Vertices in Random Planar Maps}

\author{Michael Drmota$^*$, Marc Noy\textsuperscript{\dag{}} {} and Benedikt Stufler\textsuperscript{\dag{}\dag{}} {} }

\thanks{{}$^*$ TU Wien, Institute of Discrete Mathematics and Geometry,
Wiedner Hauptstrasse 8-10, A-1040 Vienna, Austria. michael.drmota@tuwien.ac.at. Research 
supported by the
Austrian  Science Foundation FWF, project S9604.}

\thanks{\textsuperscript{\dag{}} Universitat Polit\`ecnica de Catalunya, 
Departament de Matem\`atica Aplicada II, Jordi Girona
1--3, 08034 Barcelona, Spain. marc.noy@upc.edu. Research supported in part by Ministerio de
Ciencia e Innovaci\'on MTM2008-03020.}

\thanks{\textsuperscript{\dag{}\dag{}} TU Wien, Institute of Discrete Mathematics and Geometry,
	Wiedner Hauptstrasse 8-10, A-1040 Vienna, Austria. \includegraphics{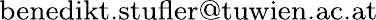}}

\begin{document}

\begin{abstract}
The main goal of this paper is to determine the asymptotic behavior of the number $X_n$
of cut-vertices in random planar maps with $n$ edges. It is shown that $X_n/n \to c$
in probability (for some explicit $c>0$). For so-called subcritical classes of planar
maps (like outerplanar maps) we obtain a central limit theorem, too.
Interestingly the combinatorics behind this seemingly simple problem is quite involved.
\end{abstract}

\maketitle

\section{Introduction}
\label{sec1}

A {\it planar map} is a connected planar graph, possibly with loops and multiple edges,
together with an embedding in the plane. A map is {\it rooted} if a vertex $v$ and an edge
$e$ incident with $v$ are distinguished, and are called the {\it root-vertex} and {\it root-edge},
respectively. Usually the root-edge is considered as directed away from the root-vertex.
In this sense, the face to the right of $e$ is called the {\it root-face} and is usually taken
as the outer face. All maps in this paper are rooted.

\begin{figure}[h]
	\centering
	\begin{minipage}{\textwidth}
		\centering
	\includegraphics[width=0.6\linewidth]{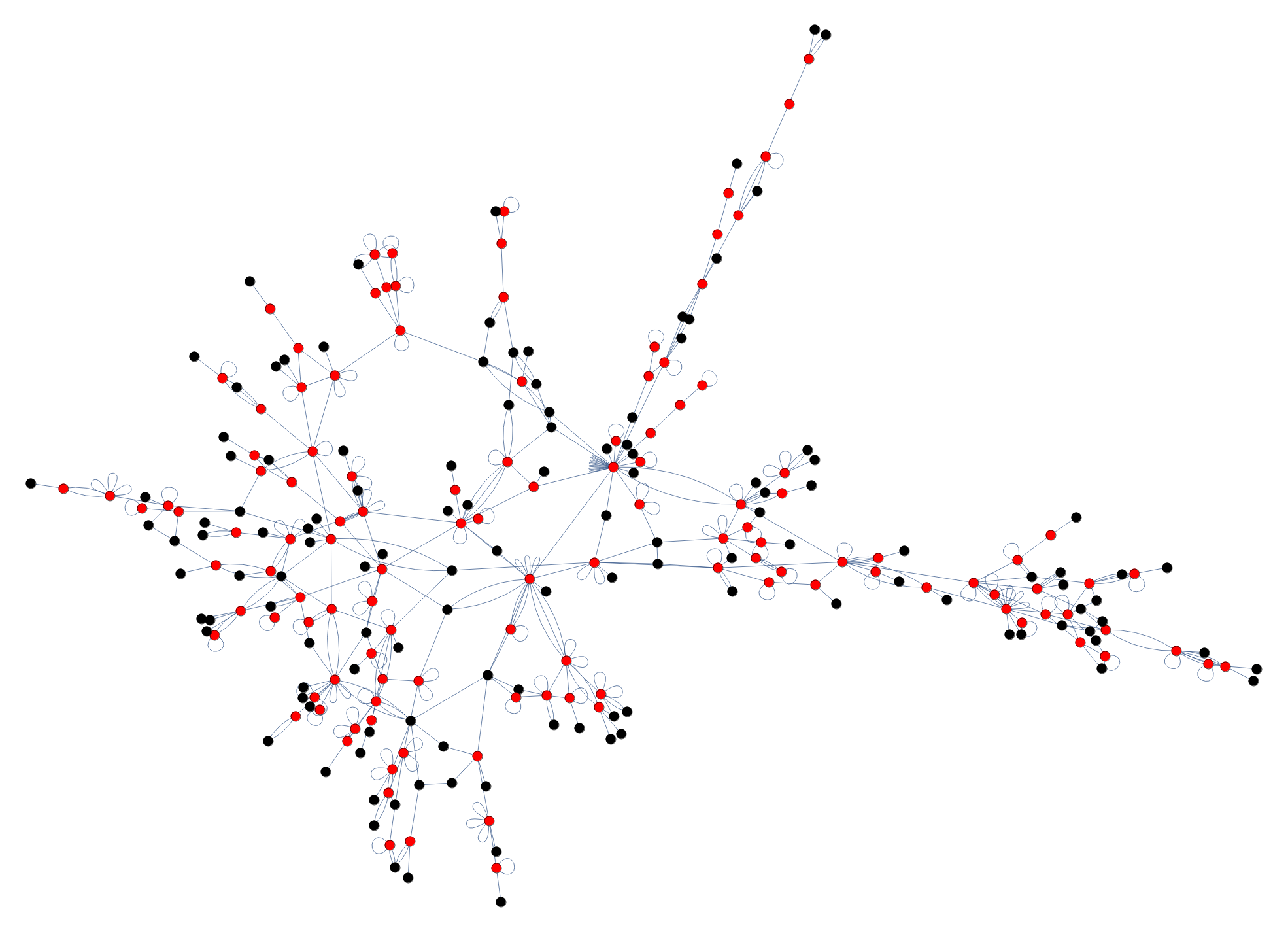}
		\caption{A randomly generated planar map with $500$ edges, embedded using a spring-electrical method. Cut vertices are coloured red.}
		\label{fi:cutv}
	\end{minipage}
\end{figure}

The enumeration of rooted maps is a classical subject, initiated by Tutte in the
1960's. Tutte (and Brown) introduced the technique now called ``the quadratic method'' in order to
compute the number $M_n$ of rooted maps with $n$ edges, proving the formula
\[
M_n = \frac{2(2n)!}{(n+2)!n!}3^n.
\]
This was later extended by Tutte and his school to several classes of planar maps:
2-connected, 3-connected, bipartite, Eulerian, triangulations, quadrangulations, etc.

The standard random model is to assume that every map of size $n$ appears with the
same probability $1/M_n$. Within this random setting several shape parameters of 
random planar maps have been studied so far, see for example~\cite{MR1871555,MR3071845, inprocdoubletri,DRMOTA2019108666}.
However, the number of cut vertices does not appear to have been studied. (A cut vertex $v$ is a vertex that disconnects a 
graph when it is removed. That is, we may partition the edge set into two non-empty classes such that $v$ is the only vertex that is incident to edges of both classes.) Figure~\ref{fi:cutv} displays a randomly generated planar map with cut vertices coloured red.  It is natural to expect that the number of cut vertices is asymptotically linear in $n$, and this is in fact true.

\begin{theo}\label{Th1}
Let $X_n$ denote the number of cut vertices in random planar maps with $n$ edges.
Then we have
\begin{equation}\label{eqTh1}
\frac {X_n}n \convp c,
\end{equation}
where
\[
c = \frac{5-\sqrt{17}}4 \approx 0.219223594
\]
In particular we have $\mathbb{E}\, X_n = c n + O(1)$. 
\end{theo}

We provide two  proofs of Theorem~\ref{Th1}. First, by a probabilistic approach that makes use of the local convergence of random planar maps re-rooted at a uniformly selected vertex (see Section~\ref{sec:local}). Second, by a  self-contained combinatorial approach
based on generating functions and singularity analysis (see Section~\ref{sec:comb}). The combinatorial approach yields additional information on related generating functions and error terms.

We conjecture that $X_n$ satisfies additionally  a normal central limit theorem. The intuition behind this is that $X_n$ may be written as the sum of $n$ seemingly weakly dependent indicator variables. The conjecture is backed up numerical simulations we carried out, see the histogram in Figure~\ref{fi:histogram}. Sampling over $2\cdot10^5$ planar maps with $n=5 \cdot 10^5$ edges, we obtained an average value of approximately $\mathbf{0.219223}677  \cdot n$ cut vertices. This value is already very close to the exact asymptotic value obtained in Theorem~\ref{Th1}. The variance  was approximately $0.082788 \cdot n$. Actually the combinatorial method presented in Section~\ref{sec:comb} can be extended to
determine the asymptotic behaviour of the variance (see Section~\ref{sec:comb-variance}). However, we will not work out the (very lengthy) details.
The main obstacle for establishing  a limit normal law is that it seems  impossible to extend this method to arbitrary moments.

We remark that the analogous problem for graphs is easier than for maps. This is because when decomposing recursively a connected graph into blocks (2-connected) components, it is straightforward  to mark whether a given vertex becomes a cut vertex or not. A central limit theorem for the number of cut vertices in random planar graphs was proved in \cite{GNR03}. 
A related parameter is the number of blocks, which also obeys a central limit theorem in the case of planar graphs \cite{GNR03}. Here we prove an analogous 
result for planar maps. The reason we can handle this parameter efficiently is that in the recursive decomposition of a planar map into blocks it is easy to keep track of the number of blocks, whereas this is not possible for the number of cut vertices: a vertex $v$ becomes a cut vertex if no corner incident with $v$ contains  a non-empty map. another situation in which vertex cuts can be handled leading to a central limit theorem is the number of 2-cuts and 3-cuts in triangulations \cite{MR2735346}.

\begin{theo}\label{th:blocks}
	Let $X_n$ denote the number of blocks in random planar maps with $n$ edges.
Then $X_n$ satisfies a central limit theorem of the form
\begin{equation}
\frac{X_n -  n/2}{\sqrt n} \convd N(0,\sigma^2)
\end{equation}
 $\sigma^2 =3/8$.
\end{theo}

\begin{figure}[h]
	\centering
	\begin{minipage}{\textwidth}
		\centering
	\includegraphics[width=0.75\linewidth]{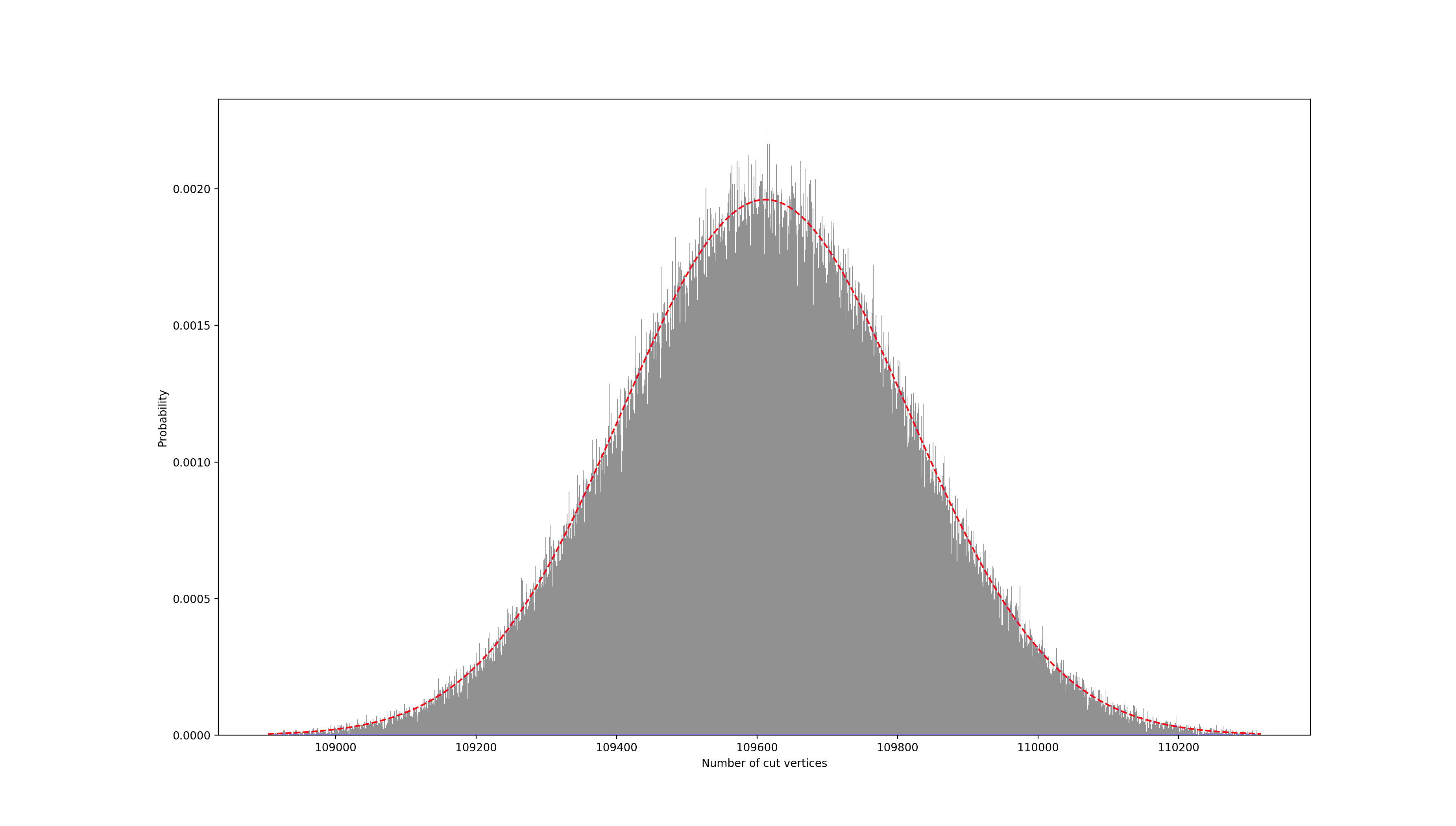}
		\caption{Histogram for the number of cut vertices in more than $2 \cdot 10^5$ randomly generated planar maps with $n=5 \cdot 10^5$ edges each.}
		\label{fi:histogram}
	\end{minipage}
\end{figure}

One important property of random planar maps that we will use in the proof of Theorem~\ref{Th1}
is that it has a {\it giant 2-connected component} of linear size. 
There are, however, several interesting subclasses of planar maps, for example series-parallel
maps,\footnote{A graph is series-parallel if does not contain $K_4$ as a minor.} where all 2-connected components are (typically)
of small size. Informally this means that on a global scale the map looks more or less like a tree. 
Such classes of maps are called subcritical; we will give a precise definition in Section~\ref{sec:gf}.
The proof of Theorem~\ref{Th2} is given in Section~\ref{sec:subcritical}.

\begin{theo}\label{Th2}
Let $X_n$ denotes the number of cut vertices in random planar maps of size $n$
in an aperiodic subcritical class of planar maps.
Then $X_n$ satisfies a central limit theorem of the form
\begin{equation}\label{eqTh2}
\frac{X_n - c n}{\sqrt n} \convd N(0,\sigma^2)
\end{equation}
where $c> 0$ and $\sigma^2 > 0$.
\end{theo}

There are some special subclasses of planar maps, where the block-decomposition is 
not {\it unrestriced}, for example, outerplanar maps, where we have to ensure that all
vertices are on the outer face. Such classes are not covered by Theorem~\ref{Th2}.
Nevertheless, they behave in several aspects like subcritical maps.
For outerplanar maps as well as for bipartite outerplanar maps we obtain a central 
limit theorem for the number of cut vertices with parameters
\[
c = \frac 14 \quad\mbox{and}\quad \sigma^2 = \frac 5{32}.
\] 
in the outerplanar case and 
\[
c = \frac{-1 + \sqrt{3}}{2} \quad \mbox{and} \quad \sigma^2 = \frac{-17 +11 \sqrt{3}}{12}
\]
in the bipartite outerplanar case.
We will discuss these examples in Section~\ref{sec:subcritical}, too.

\section{Generating Functions for Planar Maps}\label{sec:gf}

The generating function planar maps is given by
\begin{equation}\label{1.2}
M(z) = \sum_{n\ge 0} M_n z^n = \frac{18z - 1 + (1 - 12z)^{3/2}}{54z^2}
 = 1 + 2z + 9z^2 + 54z^3 + \cdots ,
\end{equation}
This can be shown in various ways, for example by the so-called quadratic method,
where it is necessary to use an additional {\it catalytic variable} $u$ that 
takes care of the root face valency. 
The corresponding generating function $M(z,u)$ ($u$ takes care of the root face valency
or equivalently by duality of the root degree) 
satisfies then 
\begin{equation}\label{1.3}
M(z, u) = 1 + zu^2M(z, u)^2 + uz
\frac{ uM(z, u) -M(z)}{u-1}
\end{equation}
which follows from a combinatorial consideration (removal of the root edge).
Then this relation can be used to obtain (\ref{1.2}) and to solve
the counting problem. We refer to \cite[Sec. VII. 8.2.]{MR2483235}.

Similarly it is possible to count also the number of non-root faces 
(with an additional variable $x$) which leads to the relation\footnote
{By abuse of notation we will use for simplicity  for  $M(z)$, $M(z,u)$, $M(z,x,u)$ the
same symbol.}
\begin{equation}\label{1.4}
M(z,x,u) = 1 + zu^2M(z,x, u)^2 + uzx
\frac{ uM(z,x, u) -M(z,x,1)}{u-1}.
\end{equation}
Note that by duality $M(z,x,1)$ can be also seen as the generating function
that is related to edges and non-root vertices of planar maps.

A planar map is 2-connected (or non-separable) if it does not contain cut vertices.
There are various ways to obtain relations for the corresponding generating
function $B(z,x,u)$ of 2-connected planar maps -- as above $z$ takes care of the number of
edges, $x$ of the number of non-root faces, and $u$ of the valency of the root face.
By using the fact, that a 2-connected planar map, where we delete the root edge, decomposes
into a sequence of 2-connected maps or single edges, we obtain the relation
\begin{equation}\label{1.5}
B(z,x,u) = zxu \frac{ \frac{uB(z,x,1)-B(z,x,u)}{1-u} + zu }{ 1- \frac{ uB(z,x,1)-B(z,x,u)}{1-u} - zu  }.
\end{equation}
We can use, for example, the quadratic method to solve this equation or we just
check that we have
\begin{align}
B(z,x,u) &= 
- \frac 12\left(1 - (1 + U - V + UV - 2U^2V)u + U(1 - V)^2u^2\right) \label{eqBexpl}  \\
&+ \frac 12(1 - (1 - V)u)
\sqrt{1 - 2U(1 + V - 2UV)u + U^2(1 - V)^2u^2},   \nonumber
\end{align}
where $U = U(x,y)$ and $V = V(x,y)$ are given by the algebraic equations
\begin{equation}\label{eqUVdef}
z = U(1-V)^2, \qquad xz = V(1-U)^2.
\end{equation}

Note that in the above counting procedure we do not take the one-edge map (nor the one-edge loop) into account.
Therefore we have to add the term $zu$ on the right hand side in order to 
cover the case of a one-edge map that might occur in this decomposition.

Sometimes it is more convenient to
include the one-edge map as well as the one-edge loop to 2-connected maps (since they have no cut-points) 
which leads us to the alternate generating function 
\[
A(z,x,u) = B(z,x,u) + zxu + zu^2.
\]
Now a general rooted planar map can be obtained from a 2-connected rooted map (including 
the one-edge map as well as the one-edge loop) by adding to every corner a rooted planar map
(note that there are $2n$ corners if there are $n$ edges):
\begin{equation}\label{1.5-2}
M(z,x,u) = 1 + A \left( z M(z,x,1)^2, x, \frac{u M(z,x,u)}{M(z,x,1)}    \right).
\end{equation}

If $x=1$ then $V(z,1)$ (and $U(z,1)$) satisfies the equation $z=V(1-V)^2$ and, thus, the dominant singularity 
of $V(z,1)$ (and $U(z,1)$) is $z_0 = \frac 4{27}$, and we also have $V(z_0,1) = \frac 13$ (as well as $U(z_0,1) = \frac 13$).
Hence, from (\ref{eqBexpl}) it follows that the function $A(z,1,1)$ has its dominant singularity at $z_0 = \frac 4{27}$, too.
On the other hand, by (\ref{1.2}) $M(z)$ has its dominant singularity at $z_1 = \frac 1{12}$ and we also have
$M(z_1) = \frac 43$. Since $z_1M(z_1)^2 = \frac 4{27} = z_0$, the singularities of $M(z)$ and $A(z,1,1)$
interact. We call such a situation {\it critical}. 

The relation (\ref{1.5-2}) can also be seen as a way how all planar maps can be constructed (recursively)
from 2-connected planar maps -- which reflects the block-decomposition of a connected graph into its 
2-connected components. More precisely, if we consider the (unique) 2-connected component that contains the
root edge -- this component might be also a one-edge map or a loop in this context -- then every vertex of 
degree $k$ in this 2-connected component is attached with $k$ rooted planar maps. 
Actually this principle holds, too, for several sub-classes of planar maps, for example for 
series-parallel planar maps. In all these cases we have a relation of the kind (\ref{1.5-2}), where
$A(z,x,u)$ is then the generating function of the corresponding 2-connected components.
Let $z_0$ denote the radius of convergence of $A(z,1,1)$ and $z_1$ the radius of convergence of $M(z,1)$.
Then the sub-class of planar maps are called {\it subcritical} if
\begin{equation}\label{eqsubcrcond}
z_1M(z_1)^2 < z_0,
\end{equation}
so that the singularities do not interact. For example, series-parallel planar maps are subcritical  
in this sense. In this case we have 
\[
A(z,1,1) = z + \frac z2 \left(1 -
z - \sqrt{1 - 6z + z^2}\right) 
\]
with radius of convergence $z_0 = 3-2\sqrt 2 \approx 0.17157$. Hence, from $M(z) = 1 + A(zM(z)^2,1,1)$ 
it follows that the radius of convergence of $M(z)$ is $z_1 \approx 0.1119109$. Furthermore $M(z_1) \approx 1.23150$,
and consequently $z_1M(z_1)^2 \approx  0.16972 < z_0$.

As already mentioned abover, there are, however, certain sub-classes of planar maps that do not fit into the scheme (\ref{1.5-2}) 
but into a very similar one. As an example we consider outerplanar maps -- these are maps, where all vertices are on the outer face.
Here the generating function $M_O(z)$ of outerplanar (rooted) maps satisfies
\begin{equation}\label{eqGFouterplanar}
M_O(z) = \frac z{1-A_O(M_O(z))},
\end{equation}
where $A_O(z)$ is the generating functions for polygon dissections (plus a single edges) where $z$
marks non-root vertices, which satisfies
\begin{equation}\label{eqGFouterplanar-2}
2 A_O(z)^2 - (1+z) A_O(z) + z =0.
\end{equation}
Note that the dominant singularity of $A_O(z)$ is $z_{0,O} = 3 - 2\sqrt 2$, whereas the
dominant singularity of $M_O(z)$ is $z_{1,O} = \frac 18$ and we have $M_O(z_{1,O}) = \frac 1{18}$.
So we clearly have 
\begin{align}
\label{eq:subcout}
M_O(z_{1,O}) < z_{0,O},
\end{align} so that the singularities of $M_O(z)$ and $A_O(z)$ do not
interact. Such a situation will be also considered as {\it subcritical}. 
It is, however, not that clear how one can take cut-vertices into account, too.
Fortunately this is possible for outerplanar maps.
Let $M_O(z,y)$ denote the generating function of outerplanar maps, where $y$ takes care of the
number of cut-vertices. It is easy to see that $M_O(z,y)$ satisfies the functional equation
\[
M_O(z,y) = \frac z{1- A_O(z + y(M_O(z,y)-z))}
\]
that reduces to (\ref{eqGFouterplanar}) if we set $y=1$. Even if we vary $y$ around $1$ 
we observe by continuity that singularities of $M_O(z,y)$ and $A_O(z)$ do not interact.


Finally we call a sub-class $M^*$ of planar maps {\it aperiodic}, if the coefficients $[z^n] M^*(z)$ 
are positive for all sufficiently large $n$. 
For example, general planar maps as well as outerplanar maps form an aperiodic sub-class.

In Section~\ref{sec:subcritical} we will prove Theorem~\ref{Th2} and will discuss then
also outerplanar and bipartite outerplanar maps.

\section{A probabilistic approach to cut vertices of planar maps}\label{sec:local}

We let $\mM_n$ denote the uniform planar map with $n$ edges. It is known that $\mM_n$ and related models of random planar maps admit a local limits that describe the asymptotic vicinity of a typical corner,  see~\cite{Stephenson2016,MR2013797,2005math.....12304K, MR3183575, MR3083919, MR3256879}.

 In a recent work by Drmota and Stufler~\cite[Thm. 2.1]{DRMOTA2019108666}, a related limit object $\mM_\infty$ was constructed that describes the asymptotic vicinity of a uniformly selected \emph{vertex} $v_n$ of $\mM_n$ instead. That is, $\mM_\infty$  is a random infinite but locally finite planar map with a marked vertex such that
\begin{align}
	\label{eq:lim1}
	(\mM_n, v_n) \convd \mM_\infty
\end{align}
in the local topology. 

In the  present section we provide a probabilistic proof of Theorem~\ref{Th1}. There are two steps. The first proves a law of large numbers for the number $X_n$ of cut vertices in $\mM_n$ without determining it explicitly:

\begin{lemma}
	\label{le:step1}
	We have $X_n/n \convp p/2$, with $p>0$ the probability that the root of $\mM_\infty$ is a cut vertex.
\end{lemma}

The factor $1/2$ origins from the fact that the number of vertices in the random map $\mM_n$ has order $n/2$. We prove Lemma~\ref{le:step1} in Section~\ref{sec:proof1} below. In the second step, we determine this limiting probability.

\begin{lemma}
	\label{le:step2}
	It holds that $p = \frac{5-\sqrt{17}}{2}$.
\end{lemma}

The proof of Lemma~\ref{le:step2} is given in Section~\ref{sec:proof2} below.

\subsection{The local topology}

We briefly  recall the background related to local limits. Consider the collection $\mathfrak{M}$ of  vertex-rooted locally finite planar maps. For all integers $k \ge 0$ we may consider the projection $U_k: \mathfrak{M} \to \mathfrak{M}$ that sends a map from $\mathfrak{M}$ to the submap obtained by restricting to all vertices with graph distance at most $k$ from the root vertex. The local topology is induced by the metric 
\[
	d_{\mathfrak{M}}(M_1, M_2) = \frac{1}{1 + \sup\{k \ge 0 \mid U_k(M_1) = U_k(M_2)\}}, \qquad M_1, M_2 \in \mathfrak{M}.
\]
It is well-known that the metric space $(\mathfrak{M}, d_{\mathfrak{M}})$ is a Polish space (that is, it is complete and separable). A limit of a sequence of vertex rooted maps in $\mathfrak{M}$ is called a local limit. The vertex rooted map $(\mM_n, v_n)$ is a random point of the space of $\mathfrak{M}$, and hence the standard probabilistic notions for different types of convergence (such as distributional convergence in~\eqref{eq:lim1}) of random points in Polish spaces apply.


\subsection{Continuity on a subset}

We consider the indicator variable \[
f: \mathfrak{M} \to \{0,1\}
\]
for the property that the root vertex is a cut vertex.
 
 Note that $f$ is not continuous: If $C_n$ denotes a cycle of length $n \ge 3$ with a fixed root vertex, then $C_n$ has no cut vertices at all. However the limit $\lim_{n \to \infty} C_n$ in the local topology is a doubly infinite path, and every vertex of this graph is a cut vertex.

Now consider the subset $\Omega \subset \mathfrak{M}$ of all locally finite  vertex-rooted maps with the property, that either the root is not a cut vertex, or it is a cut vertex and deleting it creates at least one finite connected component.

\begin{lemma}
	\label{le:conti}
	The indicator variable $f$ is continuous on $\Omega$. 
\end{lemma}
\begin{proof}
	Let $(M_n)_{n \ge 1}$ denote a sequence in $\mathfrak{M}$ with a local limit $M = \lim_{n \to \infty} M_n$ that satisfies $M \in \Omega$. If the root of $M$ is not a cut vertex, then there is a finite cycle containing it, and this cycle must then be already present in $M_n$ for all sufficiently large $n$. Hence in this case $\lim_{n \to \infty} f(M_n) = 0 = f(M)$. If the root of $M$ is a cut vertex, then  $M \in \Omega$ implies that removing it creates a finite connected component, and this component must then also be separated from the remaining graph when removing the root vertex of $M_n$ for all sufficiently large $n$. Thus,   $\lim_{n \to \infty} f(M_n) = 1 = f(M)$. This shows that $f$ is continuous on $\Omega$.
\end{proof}


\subsection{Random probability measures}

The collection $\mathbb{M}_1(\mathfrak{M})$ of probability measures on the Borel sigma algebra of $\mathfrak{M}$ is a Polish space with respect to the weak convergence topology.

For any finite planar map $M$ with $k$ vertices we may consider the uniform distribution on the $k$ different rooted versions of $M$. If the map $M$ is random, then this is a random probability measure, and hence a random point in the space  $\mathbb{M}_1(\mathfrak{M})$. In particular, the conditional law $\Pr{ (\mM_n, v_n) \mid \mM_n }$ is a random point of $\mathbb{M}_1(\mathfrak{M})$. Let $\mathfrak{L}(\mM_\infty) \in \mathbb{M}_1(\mathfrak{M})$ denote the law of the random map $\mathfrak{M}$. It follows from~\cite[Thm. 1]{quenched} that
\begin{align}
\label{eq:lim2}
\Pr{ (\mM_n, v_n) \mid \mM_n } \convp  \mathfrak{L}(\mM_\infty).
\end{align}
The explicit construction of the limit $\mM_\infty$ also entails that among the connected components created when removing any single vertex of $\mM_\infty$ at most one is infinite. In particular,
\begin{align}
	\label{eq:conti}
	\Pr{\mM_\infty \in \Omega} = 1.
\end{align}

\subsection{Proving Lemma~\ref{le:step1} using the continuous mapping theorem}
\label{sec:proof1}

Let us recall the continuous mapping theorem. The reader may consult the book by Billingsley~\cite[Thm. 2.7]{MR1700749} for a detailed proof and a general introduction to notions of convergence of  measures.
\begin{prop}[The continuous mapping theorem]
	Let $\mathfrak{X}$ and $\mathfrak{Y}$ be  Polish spaces and let $g: \mathfrak{X} \to \mathfrak{Y}$ be a measurable map. Let $D_g \subset \mathfrak{X}$ denote the subset of points where $g$ is continuous. Suppose that $X, X_1, X_2, \ldots$ are  random variables with values in $\mathfrak{X}$ that satisfy $X_n \convd X$. If $X$  almost surely takes values in $D_g$, then $g(X_n) \convd g(X)$. 
\end{prop}

For example, combining the convergence~\eqref{eq:lim1} with Lemma~\ref{le:conti} and Equation~\eqref{eq:conti} allows us to apply the continuous mapping theorem with $\mathfrak{X} = \mathfrak{M}$ and $\mathfrak{Y}=\{0,1\}$ to deduce
\begin{align}
	f(\mM_n,v_n) \convd f(\mM_\infty).
\end{align}
In other words, the probability for $v_n$ to be a cut vertex of $\mM_n$ converges toward the probability $p = \Ex{f(\mM_\infty)}$ that the root of $\mM_\infty$ is a cut vertex. Equivalently, the number of vertices $\ve(\mM_n)$ in the map $\mM_n$ satisfies
\begin{align}
	\Ex{X_n / \ve(\mM_n)} \to p.
\end{align}

Of course, it follows by the same arguments that in general for any sequence of probability measures $P_1, P_2, \ldots \in \mathbb{M}_1(\mathfrak{M})$ satisfying the weak convergence $P_n \Rightarrow \mathfrak{L}(\mM_\infty)$, the push-forward measures satisfy 
\begin{align}
\label{eq:yo}
P_nf^{-1}  \Rightarrow  \mathfrak{L}(\mM_\infty)f^{-1}.
\end{align}

Let us now consider the setting $\mathfrak{X} = \mathbb{M}_1(\mathfrak{M})$, $\mathfrak{Y} = \mathbb{R}$, and
\begin{align}
	g: \mathbb{M}_1(\mathfrak{M}) \to \mathbb{R}, \quad P \mapsto \int f \, \mathrm{d}P = P(f = 1).
\end{align}
That is, a probability measure $P \in \mathbb{M}_1(\mathfrak{M})$ gets mapped to the expectation of $f$ with respect to $P$. In other words, to the $P$-probability that the root is a cut vertex. It follows from~\eqref{eq:yo} that $g$ is continuous at the point~$\mathfrak{L}(\mM_\infty)$. Hence, using~\eqref{eq:lim2} and again the continuous mapping theorem, it follows that
\begin{align}
	\Ex{ f(\mM_n, v_n) \mid \mM_n} \convd p.
\end{align} 
As $p$ is a constant, this convergence actually holds in probability. Moreover,
\begin{align}
\Ex{ f(\mM_n, v_n) \mid \mM_n} = X_n / \ve(\mM_n).
\end{align}
The number $\ve(\mM_n)$ is known to have order $n/2$. In fact, 
\begin{align}
\frac{\ve(\mM_n) - n/2}{\sqrt{25 n/32}} \convd \mathcal{N}(0, 1)
\end{align}
with $\Exb{\ve(\mM_n)} = n/2 +1$ and $\Va{\mM_n}  = 25n/32 + O(1)$. This was shown in a lecture by Noy at the Alea-meeting 2010 in Luminy. A detailed justification may be found in~\cite[Lem. 4.1]{DRMOTA2019108666}. This allows us to apply Slutsky's theorem, yielding
\begin{align}
		X_n / n \convp  p/2.
\end{align}
We have thus completed the proof of  Lemma~\ref{le:step1}.

\subsection{Structural properties of the local limit}
We let $\mM$ denote a random map following a Boltzmann distribution with parameter $
z_1 = \frac{1}{12}.
$ That is, $\mM$ attains a finite planar map $M$ with $\co(M)$ corners with probability
\begin{align}
\Pr{\mM= M} = \frac{z_1^{\co(M)}}{M(z_1)} = \frac{3}{4} \left( \frac{1}{12}\right)^{\co(M)}.
\end{align}
The local limit $\mM_\infty$ exhibits a random number of independent copies of $\mM$ close to its root:
\begin{lemma}
	\label{le:shape}
	There is an infinite random planar map $\mM_\infty^*$ with a root vertex $u^*$ that is not a cut vertex of $\mM_\infty^*$,  such that $\mM_\infty$ is distributed like the result of attaching an independent copy of $\mM$ to each corner incident to $u^*$.
\end{lemma}
Here we use the term \emph{attach} in the sense that the origin of the root-edge of the independent copy of $\mM$ gets identified with the vertex $u^*$. The proof of Lemma~\ref{le:shape} provides additional information about the distribution of $\mM_\infty$ and  $\mM_\infty^*$. However, the only thing we are going to use and require for further arguments is the existence of such a map $\mM_\infty^*$.

\begin{proof}[Proof of Lemma~\ref{le:shape}]
	A direct description of the limit $\mM_\infty$ that uses a generalization of the Bouttier, Di Francesco and Guitter bijection~\cite{MR2097335} was given in~\cite[Thm. 1]{quenched}. Although the structure of $\mM_\infty$ may be studied in this way, it will be easier to show that $\mM_\infty$ has the desired shape via a construction related to limits of the $2$-connected core within~$\mM_n$.

	Let $\cB(\mM_n) \subset \mM_n$ denote the largest (meaning, having a maximal number of edges) $2$-connected block in the map $\mM_n$. Typically $\cB(\mM_n)$ is uniquely determined, as the number $c(n)$ of corners of $\cB(\mM_n)$ is known to have order $2n/3$, and the number of corners in the second largest block has order $n^{2/3}$.
	
	Consider the random planar map $\bar{\mM}_n$ constructed from the core $\mC_n := \cB(\mM_n)$ by attaching for each integer $1 \le i \le c(n)$ an independent copy $\mM(i)$ of $\mM$ at the $i$th corner of~$\mC_n$. We use the notation $\mC_n$ instead of $\cB(\mM_n)$ from now on to emphasize that we consider $\mC_n$ always as a part of $\bar{\mM}_n$ (as opposed to $\mM_n$).

	Clearly, the two models $\mM_n$ and $\bar{\mM}_n$  are not identically distributed. For example, the number of edges in $\bar{\mM}_n$ is a random quantity that fluctuates around $n$. However, analogously as in the proof of~\cite[Lem. 9.2]{2019arXiv190804850S}, local convergence of $\bar{\mM}_n$ is equivalent to local convergence of $\mM_n$, implying that $\mM_\infty$ is also the local limit of $\bar{\mM}_n$ with respect to a uniformly selected vertex $u_n$.
	
	The random $2$-connected planar map $\mB_n$ with $n$ edges was shown to admit a local limit $\hat{\mB}$ that describes the asymptotic vicinity of a typical corner (equivalently, the root-edge of $\mB_n$), see~\cite[Thm. 1.3]{2019arXiv190804850S}. Arguing entirely analogously as in \cite{DRMOTA2019108666}, it follows that there is  also a local limit $\mB_\infty$ that describes the asymptotic vicinity of a typical vertex.
	
	The number of vertices of $\bar{\mM}_n$ has order $n/2$, and the number of vertices in  $\mC_n$ is known to have order $n/6$.  Let $u_n^{\mathrm{B}}$ denote the result of conditioning the random vertex $u_n$ to belong to $\mC_n$. The probability for this to happen tends to $1/3$. As $u_n^{\mathrm{B}}$ is uniformly distributed among all vertices of $\mC_n$, it follows that $(\mC_n, u_n^{\mathrm{B}}) \convd \mB_\infty$ in the local topology. This implies that $(\bar{\mM}_n, u_n^{\mathrm{B}})$ converges in distribution towards the result $\mM_\infty^{\mathrm{B}}$ of attaching an independent copy of $\mM$ to each corner of $\mB_\infty$. The limit $\mM_\infty^{\mathrm{B}}$ has the desired shape. 
	
	Let $u_n^{\mathrm{c}}$ denote the result of conditioning the random vertex $u_n$ to lie outside of $\mC_n$. It remains to show that the limit $\mM_\infty^{\mathrm{c}}$ of $(\bar{\mM}_n, u_n^\mathrm{c})$ has the desired shape as well. Let $1 \le i_n \le c(n)$ denote the index of the corner where the component containing $u_n^\mathrm{c}$ is attached. It is important to note that given the maps $\mM(1), \ldots, \mM(c(n))$, the random integer $i_n$ need not be uniform, as it is more likely to correspond to a map with an above average number of vertices. This well-known waiting time paradox implies that \emph{asymptotically} the component containing $u_n^\mathrm{c}$ follows a size-biased distribution $\mM^\bullet$. That is, $\mM^\bullet$ is a random finite planar map with a marked non-root vertex, such that for any planar map $M$ with a marked non-root vertex $v$ it holds that
	\begin{align}
		\Pr{\mM^\bullet = (M,v)} = \Pr{\mM= M} / (\Ex{\ve(\mM)}-1),
	\end{align}
	with $\ve(\mM)$ denoting the number of vertices in the Boltzmann planar map $\mM$.
	
	 In detail: Given the random number $c(n)$, let $i_n^*$ be uniformly selected among the integers from $1$ to $c(n)$. For each $1 \le i \le c(n)$ with $i \ne i_n^*$ let $\bar{\mM}(i)$ denote an independent copy of $\mM$, and let $\bar{\mM}(i_n^*)$ denote an independent copy of $\mM^\bullet$. Likewise, for each $1 \le i \le c(n)$ with $i \ne i_n$ set $\mM*(i) = \mM(i)$, and let $\mM^*(i_n) = (\mM(i_n), u_n^\mathrm{c})$. Analogously as in the proof of ~\cite[Lem. 9.2]{2019arXiv190804850S}, it follows that 
	 \begin{align}
	 	(\mM^*(i))_{1 \le i \le c(n)} \atv (\bar{\mM}(i))_{1 \le i \le c(n)}.
	 \end{align}
	This entails that the core $\mC_n$ rooted at the corner with index $i_n$ admits $\hat{\mB}$ (and not $\mB_\infty$) as local limit. Moreover, the local limit $\mM_\infty^{\mathrm{c}}$ of $\bar{\mM}_n$ rooted at $u_n^\mathrm{c}$ may be constructed by attaching an independent copy of $\mM$ to each corner of $\hat{\mB}$, except for the root-corner of $\hat{\mB}$, which receives an independent copy of $\mM^\bullet$. The marked vertex of the limit object $\mM_\infty^{\mathrm{c}}$ is then given by the marked vertex of this component.
	
	To proceed, we need information on the shape of $\mM^\bullet$. Consider the ordinary generating functions $M(v,w)$ and $A(v,w)$ of planar maps and $2$-connected planar maps, with $v$ marking corners, and $w$ marking non-root vertices. The block-decomposition yields
	\begin{align}
		\label{eq:blockmap}
		M(v,w) = A(vM(v,w),w).
	\end{align}
	That is, a planar map consists of a uniquely determined block containing the root-edge, with uniquely determined components attached to each of its corners. Let us call this block the \emph{root block}. For the trivial map consisting of a single vertex and no edges, this block is identical to the trivial map, with nothing attached to it as it has no corners.

	Marking a non-root vertex (and no longer counting it) corresponds to taking the partial derivative with respect to $w$. It follows from~\eqref{eq:blockmap} that
	\begin{align}
		\label{eq:vrooted}
		\frac{\partial M}{\partial w}(v,w) = \frac{\partial A}{\partial w}(vM(v,w),w) + \frac{\partial A}{\partial v}(vM(v,w),w) v \frac{\partial M}{\partial w}(v,w).
	\end{align}
	The combinatorial interpretation is that either the marked non-root vertex is part of the root block (accounting for the first summand), or there is a uniquely determined corner of the root block such that the component attached to this corner contains it. This is a recursive decomposition, as in the second case we could proceed with this component, considering whether the marked vertex belongs to its root block or not. We may do so a finite number of times, until it finally happens that the marked vertex belong to the root-block of the component under consideration. That is, if we follow this decomposition until encountering the marked non-root vertex, we have to pass through a uniquely determined sequence of blocks, always proceeding along uniquely determined (and hence marked) corners, until arriving at a block with a marked non-root vertex. On a generating function level, this is expressed by
	\begin{align}
		\label{eq:unrolled}
		\frac{\partial M}{\partial w}(v,w)  = \frac{1}{1- \frac{\partial A}{\partial v}(vM(v,w),w) v}\frac{\partial A}{\partial w}(vM(v,w),w).
	\end{align}
 This allows us to apply Boltzmann principles, yielding that the random map $\mM^\bullet$ may be sampled in two steps, that may be described as follows: First, generate this sequence of blocks by linking a geometrically distributed random number $N$ of random independent Boltzmann distributed blocks $\mB_1^\circ, \ldots, \mB_N^\circ$ with marked corners into a chain, and attach an extra random Boltzmann distributed block $\mB^\bullet$ with a marked non-root vertex to the end of the chain. The random number $N$ has generating function
 \begin{align}
 	\Ex{u^N} = \frac{1- \frac{\partial A}{\partial v}(z_1 M(z_1,1),1) z_1}{1- u \frac{\partial A}{\partial v}(z_1 M(z_1,1),1) z_1}.
 \end{align}
 The corner-rooted blocks are independent copies of a Boltzmann distributed block $\mB^\circ$, whose number of corners $\co(\mB^\circ)$ has generating function
 \begin{align}
 	\Ex{u^{\co(\mB^\circ)}} = \frac{\frac{\partial A}{\partial v}(u z_1 M(z_1 ,1),1)}{\frac{\partial A}{\partial v}( z_1 M(z_1 ,1),1)}.
 \end{align}
 The distribution of $\mB^\circ$ is fully characterized by the fact that, when conditioning on the number of corners, $\mB^\circ$ is conditionally uniformly distributed among the corner-rooted blocks with that number of corners. The distribution of $\mB^\bullet$ is defined analogously. If we attach a block $\tilde{B}$ to the marked corner $c$ of some block $B$, we say the resulting corner ``to the right'' of $\tilde{B}$ \emph{corresponds} to $c$. Hence the map obtained by linking $(\mB_1^\circ, \ldots, \mB_N^\circ, \mB^\bullet)$ has precisely $N$ corners that correspond marked corners. We call these corners \emph{closed}, and all other corners \emph{open}. The second and final step in the sampling procedure of $\mM^\bullet$ is to attach an independent copy of $\mM$ to each open corner of the map corresponding to $(\mB_1^\circ, \ldots, \mB_N^\circ, \mB^\bullet)$. Note that since the marked vertex of $\mB^\bullet$ is a non-root vertex, all corners incident to the marked vertex are open. Consequently, the limit $\mM_\infty^{\mathrm{c}}$ has the desired shape, and the proof is complete.
\end{proof}

\subsection{Proving Lemma~\ref{le:step2} via the asymptotic  degree distribution}
\label{sec:proof2}

Let $q(z) = \sum_{k \ge 1} q_k z^k$ denote the probability generating function of the root-degree of the map $\mM_\infty^*$. If we attach an independent copy of $\mM$ to each corner incident to the vertex $u^*$ in the map $\mM_\infty^*$, then $u^*$ becomes a cut vertex if and only if at least one of these copies has at least one edge. The probability for $\mM$ to have no edges, that is, to consist only of a single vertex, is given by $1 / M(z_1) = 3/4$. Hence the probability $p$ for the root of $\mM_\infty$ to be a cut vertex may be expressed by
\begin{align}
	\label{eq:expforp}
	p = \sum_{k \ge 1} q_k \left(1 - \left(\frac{3}{4}\right)^k\right) = 1 - q\left(\frac{3}{4}\right).
\end{align}

Hence, in order to determine $p$ we need to determine $q(z)$. Surprisingly, we may do so without concerning ourselves with the precise construction of $\mM_\infty^*$.

It was shown in~\cite{zbMATH00683269} that the degree of the origin of the root-edge  of the random planar map $\mM_n$ admits a limiting distribution with a generating series $d(z)$ given by
\begin{align}
d(z) = \frac{z \sqrt{3}}{ \sqrt{(2+z)(6-5z)^3}}.
\end{align}
That is, $d_k := [z^k] d(z)$ is the asymptotic probability for the origin of the root-edge of $\mM_n$ to have degree $k$. Let $s_k$ denote the limit of the probability for a uniformly selected vertex of $\mM_n$ to have degree $k$. It follows from~\cite[Prop. 2.6]{MR1666953} that 
\begin{align}
\label{eq:relpd}
s_k = 4 d_k / k
\end{align}
for all integers $k \ge 1$.  Setting $s(z) = \sum_{k \ge 1} s_k z^k$, Equation~\eqref{eq:relpd} may be rephrased by
\begin{align}
zs'(z) = 4 d(z).
\end{align}
Via integration, this yields the expression
\begin{align}
s(z) =  
\frac{1}{2} \left(-1+\frac{\sqrt{2+z}}{\sqrt{2-\frac{5 z}{3}}}\right)
\end{align}

As $\mM_\infty$ is the local limit of $\mM_n$ rooted at a uniformly chosen vertex, it follows that for each $k \ge 1$ the limit $s_k$ equals the probability for the root of $\mM_\infty$ to have degree $k$. Let $r(z)$ denote the probability generating series of the degree distribution of the origin of the root-edge of the Boltzmann map $\mM$. It follows from Lemma~\ref{le:shape} that
\begin{align}
\label{eq:doit}
s(z) = q(zr(z)).
\end{align}

We are going to compute $r(z)$. To this end, let $M(z,v)$ denote the generating series of planar maps with $z$ marking edges and $v$ marking the degree of the root vertex. By duality, $M(z,v)$ coincides with the bivariate generating series where the second variable marks the degree of the outer face. The quadratic method~\cite[p. 515]{MR2483235} hence yields the known expression
\begin{align}
M(z_1,u)  = \frac{-3u^2+36u -36 + \sqrt{3(u+2)(6-5u)^3}}{6u^2(u-1)}.
\end{align}
The series $r(z)$ is related to $M(z,u)$ via
\begin{align}
r(u) = M(z_1,u) / M(z_1,  1) = \frac{3}{4} M(z_1,u).
\end{align} 	
Forming the compositional inverse of $zr(z)$ and plugging it into Equation~\eqref{eq:doit} yields the involved expression 
\begin{align}
\label{eq:expforq}
q(z) = \frac{1}{2} \left(\frac{\sqrt{\frac{20 z^2+48
			z-\sqrt{2 z-27} (2 z-3)^{3/2}+123}{z (4
			z+3)+24}}}{2 \sqrt{\frac{6-4 z}{-14 z+5
			\sqrt{2 z-27} \sqrt{2 z-3}+51}}}-1\right).
\end{align}
The first couple of terms are given by
\begin{align}
q(z) = \frac{4 z}{9}+\frac{56 z^2}{243}+\frac{848 z^3}{6561}+\frac{13408 z^4}{177147}+\frac{217664 z^5}{4782969}+\ldots.
\end{align}
Equation~\eqref{eq:expforq} allows us to evaluate the constant $q(3/4)$ in the expression for $p$ given in Equation~\eqref{eq:expforp}, yielding
\begin{align}
\label{eq:yo44}
p = 1 - q(3/4) = \frac{5 - \sqrt{17}}{2}.
\end{align}
This concludes the proof of Lemma~\ref{le:step2}.

\section{A combinatorial approach to cut vertices of planar maps}\label{sec:comb}

The goal of this section is to derive the constant $c$ in Theorem~\ref{Th1} with 
the help of a combinatorial approach to the cut vertex counting problem.

\subsection{More on generating functions of $2$-connected planar maps}

First we introduce (formally) a generating function that takes care of all vertex degrees
in 2-connected planar maps (including the one-edge map and the one-edge loop)
\[
\overline A(z;w_1,w_2,w_3,w_4, \ldots; u),
\]
where $w_k$, $k\ge 1$, corresponds to vertices of degree $k$ and we also take the root vertex
into account. As usual, $u$ corresponds to the root degree.

Similarly we introduce a variant of this generation function that takes care of all vertex degrees
in 2-connected planar maps (without the one-edge map and one-edge loop) and does not take the root
vertex into account:
\[
\overline B(z; w_2,w_3,w_4, \ldots; u).
\]

We recall that $A(z,x,1)$ corresponds to 2-connected maps (including the one-edge map and the one-edge loop), where $x$ 
takes non-root faces into account. By adding the factor $x$ we also include the root face 
and by duality $x A(z,x,1)$ is also the generating function, where $x$ corresponds to 
vertices.

Although it seems to be impossible to work directly with $\overline A(z;w_1,w_2,w_3, \ldots)$
or with $\overline B(z; w_2,w_3,w_4, \ldots; u)$, we 
have the following easy relations:
\begin{equation}\label{1.8}
\overline A(z;xv,xv^2,xv^3, \ldots; u) = x A(zv^2,x,u)
\end{equation}
and
\begin{equation}\label{1.8-2}
\overline B(z;xv,xv^2,xv^3, \ldots; u) = B(zv^2,x,u/v)
\end{equation}
This follows from the fact that every vertex of degree $k$ corresponds to $k$ half-edges.
So summing up these half-edges we get twice the number of edges.
In particular by taking derivatives with respect to $x$ and $v$ it follows that 
\[
\sum_{k\ge 1} \overline A_{w_k} (z;v,v^2,v^3, \ldots) v^k = A(zv^2,1,1) +  A_x(zv^2,1,1)
\]
and 
\[
\sum_{k\ge 1} k \overline A_{w_k} (z;v,v^2,v^3, \ldots)  v^{k-1} = 2zv A_z(zv^2,1,1).
\]

We also mention that 
\begin{align*}
\overline B(z;v^2,v^3,\ldots,1 ) &
= \overline A(z;v,v^2,\ldots, 1/v) - zv-z \\
&= A(zv^2,1,1/v) - zv-z \\
&= B(zv^2,1,1/v)
\end{align*}
as it should be according to (\ref{1.8-2}). 

It turns out that we will also have to deal with the sum
\[
\sum_{k\ge 1} \overline A_{w_k} (z;v,v^2,v^3, \ldots)
\]
which is slightly more difficult to understand. 

\begin{lemma}\label{Le1}
Let $u_1(z)$ denote the function $u_1(z) = 1/(1-V(z,1)$, where $V(z,x)$ (and $U(z,x)$) is given by
(\ref{eqUVdef}). Then we have
\begin{align*}
\sum_{k\ge 1} \overline A_{w_k} (z;v,v^2,v^3, \ldots)
 &= 2zv + z + B(zv^2,1,1/v) \\
 &+ zv \frac{ \frac{u_1(zv^2)B(zv^2,1,1/v)-B(zv^2,1,u_1(zv^2))/v}{1/v-u_1(zv^2)} + zvu_1(zv^2) }
{ 1- \frac{ u_1(zv^2)B(zv^2,1,1/v)-B(zv^2,1,u_1(zv^2))/v}{1/v-u_1(zv^2)} - zvu_1(zv^2)  }
\end{align*}
\end{lemma}
Note that some simplifications in this representations are possible.
For example we have 
\[
B(zv^2,1,u_1(zv^2)) = V(zv^2,1)^2.
\]

\begin{proof}
We note that the derivative with respect to $w_k$ 
marks a vertex of degree $k$ and discounts it. By substituting $w_k$ by $v^k$ we, thus, see that
the resulting exponent of $v$ is twice the number of edges minus the degree of the marked vertex.
Hence we have to cover the situation, where we mark a vertex and keep track of the degree
of the marked vertex.

Let $B^\bullet(z,x,u,w)$ be the generating function of vertex marked 2-connected planar maps, 
where the marked vertex is different from the root and where $u$ takes care of
the root degree and $w$ on the degree of the pointed vertex. By duality this is also the generating 
function of face marked 2-connected planar maps, where $u$ takes care of the root face valency and
$w$ of the valency of the marked face (that is different from the root face).
Then we have
\begin{equation}\label{eqsum3}
\sum_{k\ge 1} \overline A_{w_k} (z;v,v^2,v^3, \ldots) = 2zv +z + B(zv^2,1,1/v) + B^\bullet(zv^2,1,1,1/v).
\end{equation}
The term $2zv$ corresponds to the one-edge map, the term $z$ to the one-edge loop, 
the term $B(zv^2,1/v)$ to the case, where the root
vertex is marked and the third term $B^\bullet(zv^2,1,1,1/v)$ to the case, where a vertex different
from the root is marked. Note that the substitution $u = 1/v$ (or $w = 1/v$) discounts the degree
of the marked vertex in the exponent of $v$ as needed.

Thus, it remains to get an expression for $B^\bullet(z,1,u,w)$. For this purpose we start with
the generating function $B(z,1,u)$ and determine first the generating function
$\tilde B(z,x,u,w)$ (for $x=1$), where the additional variable $w$ takes care of the valency of the second face 
incident to the root edge. By using the same construction as above we have
\[
\tilde B(z,1,u,w) = zuw \frac{ \frac{uB(z,1,w)-wB(z,1,u)}{w-u} + zuw }{ 1- \frac{ uB(z,1,w)-wB(z,1,u)}{w-u} - zuw  }.
\]
This gives (by again applying this construction)
\[
B^\bullet(z,1,u,w) = \tilde B(z,1,u,w) + zu  
\frac { \frac{uB^\bullet(z,1,1,w)-B^\bullet(z,1,u,w)}{1-u}}
{\left( 1- \frac{ uB(z,1,1)-B(z,1,u)}{1-u} - zu    \right)^2 }.
\]
This equation can be solved with the help of the kernel method. By rewriting it to
\begin{align*}
& B^\bullet(z,1,u,w) \left( 1 + \frac{zu}{1-u} \frac 1{\left( 1- \frac{ uB(z,1,1)-B(z,1,u)}{1-u} - zu    \right)^2 } \right) \\
&= \tilde B(z,1,u,w) + \frac{zu^2 B^\bullet(z,1,1,w)}{1-u}\frac 1{\left( 1- \frac{ uB(z,1,1)-B(z,1,u)}{1-u} - zu    \right)^2 }.
\end{align*}
Let $u_1(z)$ be defined by the equation
\begin{equation}\label{equ1z}
 1 + \frac{zu_1(z)}{1-u_1(z)} \frac 1{\left( 1- \frac{ u_1(z)B(z,1,1)-B(z,1,u_1(z))}{1-u_1(z)} - zu_1(z)    \right)^2 } = 0
\end{equation}
Then it follows that 
\[
B(z,1,u_1(z),w) + \frac{zu_1(z)^2 B^\bullet(z,1,1,w)}{1-u_1(z)}\frac 1{\left( 1- \frac{ u_1(z)B(z,1,1)-B(z,1,u_1(z))}{1-u_1(z)} - zu_1(z)    \right)^2 }
= 0
\]
or
\begin{align}
B^\bullet(z,1,1,w) &= \frac{\tilde B(z,1,u_1(z),w)}{u_1(z)}  \label{eqBbullet} \\
&=zw \frac{ \frac{u_1(z)B(z,1,w)-w B(z,1,u_1(z))}{w-u_1(z)} + z w u_1(z) }
{ 1- \frac{ u_1(z)B(z,1,w)-w B(z,1,u_1(z))}{ w -u_1(z)} - z w u_1(z) }.
\nonumber
\end{align}
By using (\ref{eqBexpl}) and (\ref{eqUVdef}) it is a nice (but tedious) exercise to show that $u_1(z) = 1/(1-V(z,1)$.
Note that $u_1(z)$ satisfies the cubic equation $u_1(z) = 1 + zu_1(z)^3$. Thus, $u_1(z)$ is also the
generating function of ternary rooted trees.
\end{proof}

\subsection{Cut Vertices in Random Planar Maps}\label{sec:4.2}

Let $M_0(z,y)$ denote the generating function of planar maps with at least one edge, where the root vertex is 
not a cut point and where $z$ takes care of the number of edges and $y$ of the number of 
cut-points (that are then different from the root vertex). 

Next let $M_r(z,y)$ denote the generating function of (all) planar maps, where
 $z$ takes care of the number of edges and $y$ of the number of 
non-root cut-points. 

Finally let $M_a(z,y)$ denote the generating function of (all) planar maps, where
 $z$ takes care of the number of edges and $y$ of the number of 
(all) cut-points. 

Obviously we have the following relation between these three generating functions:
\begin{equation}\label{1.6}
M_a(z,y) = y M_r(z,y) - (y-1)(1+ M_0(z,y)).
\end{equation}
Note that $M_a(z,1) = M_r(z,1) = M(z)$.

Furthermore we set 
\[
E_a(z) = \left.\frac{\partial M_a(z,y)}{\partial y} \right|_{y=1} \quad \mbox{and}\quad
E_r(z) = \left.\frac{\partial M_r(z,y)}{\partial y} \right|_{y=1}. 
\]
Clearly, the generating function $E_a(z)$ is related to the expected number $\mathbb{E}[X_n]$ of
cutpoints:
\[
E_a(z) = \sum_{n\ge 0} M_n \mathbb{E}[X_n] z^n.
\]
Our first main goal is to obtain relations for $E_a(z)$ which will enable us to 
obtain asymptotics for $\mathbb{E}[X_n]$.

By differentiating (\ref{1.6}) with respect $y$ and setting $y=1$ we obtain
\[
E_a(z) = E_r(z) + M(z) - 1- M_0(z,1).
\]

With the help of the above notions we obtain the following (formal relation):
\begin{equation}\label{1.8-2aa}
M_a(z,y) = 1 + 
\overline A \left( z; y M_r(z,y) -y +1, yM_r(z,y)^2 -y +1, \ldots; 1 \right).
\end{equation}
The right hand side is based on the block-decomposition (similarly to (\ref{1.5-2})) and 
takes care, whether the vertices of the block that contains the root edge become cut-vertices or not.

Similarly we obtain
\begin{equation}\label{1.8-3}
M_0(z,y) = \overline B \left( z; y M_r(z,y)^2 -y +1, yM_r(z,y)^3 -y +1, \ldots; 1 \right) + z(y M_r(z,y) -y +1) + z.
\end{equation}
In particular if we set $y=1$ we obtain
\[
M_0(z,1) = \overline B \left( z; M(z)^2, M(z)^3 , \ldots; 1 \right) = 
B(zM(z)^2,1,1/M(z)) + zM(z)+z.
\]
This now gives
\begin{equation}\label{1.7}
E_a(z) = E_r(z) + M(z) - 1- B(zM(z)^2,1,1/M(z))-zM(z)-z.
\end{equation}

By differentiating (\ref{1.8-2aa}) with respect to $y$ and setting $y=1$ we, thus, obtain
\begin{align*}
E_a(z) &= \sum_{k\ge 1} \overline A_{w_k} \left( z; M(z), M(z)^2, \ldots; 1 \right) \\
& \qquad \times \left(  M(z)^k  -1 + k M(z)^{k-1} E_r(z)   \right) \\
&=  \sum_{k\ge 1}\overline A_{w_k} \left( z; M(z), M(z)^2, \ldots \right) M(z)^k \\
&- \sum_{k\ge 1}\overline A_{w_k} \left( z; M(z), M(z)^2, \ldots \right) \\
&+ E_r(z) \sum_{k\ge 1} k \overline A_{w_k} \left( z; M(z), M(z)^2, \ldots \right) M(z)^{k-1}.
\end{align*}
Note that 
\[
\sum_{k\ge 1} \overline A_{w_k} (z;M(z),M(z)^2, \ldots) M(z)^k =  A(zM(z)^2,1,1) +  A_x(zM(z)^2,1,1),
\]
\[
\sum_{k\ge 1} k \overline A_{w_k} (z;M(z),M(z)^2, \ldots) M(z)^{k-1} = 2 zM(z) A_z(zM(z)^2,1,1),
\]
whereas 
\begin{align*}
&\sum_{k\ge 1} \overline A_{w_k} (z;M(z),M(z)^2, \ldots) \\ &= 
2zM(z)+z + B(zM(z)^2,1,1/M(z)) + B^\bullet(zM(z)^2,1,1,1/M(z))\\
&= 2zM(z)+z + B(zM(z)^2,1,1/M(z)) \\
&+ zM(z) \frac{ \frac{u_1(zM(z)^2)B(zM(z)^2,1,1/M(z))-B(zM(z)^2,1,u_1(zM(z)^2))/M/z)}{1/M(z)-u_1(zM(z)^2)} + zM(z)u_1(zM(z)^2) }{ 1- \frac{ u_1(zM(z)^2)B(zM(z)^2,1,1/M(z))-B(zM(z)^2,1,u_1(zM(z)^2))/M(z)}{1/M(z)-u_1(zM(z)^2)} - zM(z)u_1(zM(z)^2)  }
\end{align*}
This finally leads to the explicit formula for $E_a(z)$:
\begin{align}
E_a(z) & = \frac 1{ 1- 2 zM(z) A_z(zM(z)^2,1,1)} \label{eqEcfinal}   \\ 
& \quad \times \biggl[ 
A(zM(z)^2,1,1) +  A_x(zM(z)^2,1,1) \nonumber \\
& \qquad  
- 2zM(z)-z - B(zM(z)^2,1,1/M(z)) - B^\bullet(zM(z)^2,1,1,1/M(z)) \nonumber \\  
& \qquad  
+ 2zM(z)A_z(zM(z)^2,1,1)\left( B(zM(z)^2,1,1/M(z)) -M(z) + zM(z)+z+ 1   \right)  \biggr],   \nonumber
\end{align}
where 
\begin{align*}
&B^\bullet(zM(z)^2,1,1,1/M(z)) \\&= 
zM(z) \frac{ \frac{u_1(zM(z)^2)B(zM(z)^2,1,1/M(z))-B(zM(z)^2,1,u_1(zM(z)^2))/M/z)}{1/M(z)-u_1(zM(z)^2)} + zM(z)u_1(zM(z)^2) }{ 1- \frac{ u_1(zM(z)^2)B(zM(z)^2,1,1/M(z))-B(zM(z)^2,1,u_1(zM(z)^2))/M(z)}{1/M(z)-u_1(zM(z)^2)} - zM(z)u_1(zM(z)^2)  }
\end{align*}

\subsection{Asymptotics}

We start with a proper representation of $B_x(z,1,1)$ and $B_z(z,1,1)$.

\begin{lemma}\label{Le2}
Let $B(z,x,u)$ be given by (\ref{eqBexpl}) and $u_1(z) = 1/(1-V(z,1))$ as in Lemma~\ref{Le1}. Then we have
\begin{equation}\label{eqLe21}
B_x(z,1,1) = \frac{u_1(z)-1}{u_1(z)}Q(z)(1-Q(z))
\end{equation}
and
\begin{equation}\label{eqLe22}
B_z(z,1,1) = \frac{u_1(z)-1}{z\,u_1(z)}Q(z)(1-Q(z)) + u_1(z)-1
\end{equation}
where $Q(z)$ abbreviates
\[
Q(z) =  \frac{V(z,1)^2}{u_1(z)-1}  - \frac{u_1(z) B(z,1,1)}{u_1(z)-1} + z\, u_1(z).
\]
\end{lemma}

\begin{proof}
Set 
\[
Q_0(z,x,z) = \frac{uB(z,x,1)-B(z,x,u)}{1-u} + zu
\]
Then (\ref{1.5}) rewrites to 
\[
B(z,x,u) = zxu \frac{ Q_0(z,x,u)}{1- Q_0(z,x,u)}.
\]
Hence, by taking the derivative with respect to $x$ (and then setting $x=1$) we obtain
\[
B_x(z,1,u) = zu \frac{ Q_0(z,1,u)}{1- Q_0(z,1,u)} + zu \frac{  \frac{uB_x(z,1,1)-B_x(z,1,u)}{1-u}  }{(1-Q_0(z,1,u))^2}
\]
or
\[
B_x(z,1,u) \left( 1 + \frac{zu}{(1-u)(1-Q_0(z,1,u)^2 } \right) = \frac{zu Q_0(z,1,u) }{1- Q_0(z,1,u)} 
+ \frac{zu^2 B_x(z,1,1)}{(1-u)(1-Q_0(z,1,u))^2}.
\]
If we replace $u$ by $u_1(z)$ then by (\ref{equ1z}) the left hand side vanished and, thus, the right hand side, too.
From that we obtain the explicit representation (\ref{eqLe21}) for $B_x(z,1,1)$. We just note that
\[
Q(z) = Q_0(z,1,u_1(z))
\]
since -- by (\ref{eqBexpl}) and by $u_1(z) = 1/(1-V(z,1))$ -- $B(z,1,u_1(z)) = V(z,1)^2$.

Similarly we obtain a representation for $B_z(z,1,1)$. Instead of taking the derivative with respect to $x$
we take the derivative with respect to $z$ and get
\[
B_z(z,1,u) = u \frac{ Q_0(z,1,u)}{1- Q_0(z,1,u)} + zu \frac{  \frac{uB_z(z,1,1)-B_z(z,1,u)}{1-u} +u }{(1-Q_0(z,1,u))^2}
\]
or
\[
B_z(z,1,u) \left( 1 + \frac{zu}{(1-u)(1-Q_0(z,1,u)^2 } \right) = \frac{u Q_0(z,1,u) }{1- Q_0(z,1,u)} 
+ \frac{zu^2}{(1-Q_0(z,1,u))^2} \left( \frac {  B_z(z,1,1)}{1-u } +1 \right).
\]	
Again by replacing $u$ by $u_1(z)$ the vanishing right hand side leads to (\ref{eqLe22}), the proposed
explicit representation for $B_z(z,1,1)$.
\end{proof}

This leads us the following local expansions.

\begin{lemma}\label{Le3}
We have the following local expansions in powers of $\left( 1  - \frac {27}4 z \right)$:
\begin{align}
B_x(z,1,1) &= \frac 2{27} - \frac{2\sqrt 3}{27} \sqrt{1-\frac {27}4 z} + \frac 2{81} \left( 1- \frac{27}4 z \right)
+ \frac{19\sqrt 3}{729}  \left( 1- \frac{27}4 z \right)^{3/2} + \cdots  \label{eqLe31} \\
B_z(z,1,1) &= 1- \sqrt 3 \left( 1- \frac{27}4 z \right)^{1/2} + \frac 43 \left( 1- \frac{27}4 z \right)
- \frac{35\sqrt 3}{54} \left( 1- \frac{27}4 z \right)^{3/2} + \cdots \label{eqLe32}\\
B^\bullet(z,1,1,w) &=  -4\,{\frac {w \left( -2\,w+\sqrt {4\,{w}^{2}-60\,w+81}-9 \right) }{243
-54\,w+27\,\sqrt {4\,{w}^{2}-60\,w+81}}}  \label{eqLe33}   \\
&  +{\frac {16\,\sqrt {3}{w}^{2}
 \left( -2\,w+\sqrt {4\,{w}^{2}-60\,w+81}+3 \right)}{9\, \left( 9-2
\,w+\sqrt {4\,{w}^{2}-60\,w+81} \right) ^{2} \left( 2\,w-3 \right) }} \sqrt{1-\frac {27}4 z} + \cdots  \nonumber
\end{align}
\end{lemma}

\begin{proof}
By inverting the equation $z = V(1-V)^2$ it follows that $V(z,1)$ has the local expansion
\[
V(z,1) = \frac 13 - \frac 2{3\sqrt 3} Z +   \frac 2{27} Z^2
-  \frac 5{81\sqrt{3}} Z^3 +  \cdots,
\]
where $Z$ abbreviates
\[
Z = \sqrt{1- \frac {27}4 z}.
\]
Consequently $u_1(z) = 1/(1-V(z,1))$ is given by
\[
u_1(z) = \frac 32 - \frac{\sqrt 3}2 Z + \frac 23 Z^2 
- \frac{35\sqrt{3}}{108} Z^3  \cdots
\]
We already know that
\[
B(z,1,u_1(z)) = V(z,1)^2 = {\frac{1}{9}}-{\frac {4\,\sqrt {3}}{27}}Z+{\frac{16}{81}}{Z}^{2}-{
\frac {34\,\sqrt {3}}{729}}{Z}^{3}+ \cdots
\]
and from (\ref{eqBexpl}) we directly obtain
\[
B(z,1,1) = {\frac{1}{27}}-{\frac{4}{27}}{Z}^{2}+{\frac {8\,\sqrt {3}}{81}}{Z}^{3} + \cdots
\]
Hence, the local expansion of $Q(z) = Q_0(z,1,u_1(z))$ can be easily calculated:
\[
Q(z) = \frac 13 - \frac{2\sqrt 3}{9}  Z + \frac 2{27} Z^2  
- \frac{5\sqrt 3}{243} Z^3 + \cdots,
\]
and, thus, (\ref{eqLe31}) and (\ref{eqLe32}) follow from this expansion and from (\ref{eqLe21}) and (\ref{eqLe22}). 

Finally we have to use (\ref{eqBbullet}) and the expansion for $B(x,1,w)$ to obtain (\ref{eqLe33}).
\end{proof}

This leads us to the following local expansion for $E_a(z)$ and 
a corresponding asymptotic relation.

\begin{lemma}\label{Le4}
The function $E_a(z)$ has the following local expansion
\begin{equation}\label{eqLe41}
E_a(z) = \frac{11\sqrt{17}-37}{24} - (5-\sqrt{17}) \sqrt{1-12 z} + \cdots
\end{equation}
which implies
\[
\mathbb{E}\, X_n = \frac {[z^n]\, E_a(z)}{[z^n]\, M(z)   } \sim  \frac{(5-\sqrt{17})}4 n.
\]
\end{lemma}

\begin{proof}
We note that several parts of (\ref{eqEcfinal}) have a dominant singularity of the form $(1-12z)^{3/2}$. 
For those parts only the value at $z_1=1/12$ influences the the constant term and coefficient
of $\sqrt{1-12z}$ in the local expansion of $E_a(z)$. In particular we have
\begin{align*}
M(z_1) &= \frac 43, \\
A(z_1M(z_1)^2,1,1) &= \frac 13, \\
B(z_1M(z_1)^2,1,1/M(z_1)) &= \frac{3\sqrt{17} - 11}{72}.
\end{align*}
The other functions appearing will have a non-zero coefficient at the $\sqrt{1-12z}$--term.
Note also that we have
\[
\sqrt{1-\frac {27}4 zM(z)^2} = \sqrt{3}\sqrt{1-12z} + O(|1-12 z|),
\]
Hence we get
\begin{align*}
A_z(zM(z)^2,1,1) &= 3- 3 \sqrt{1-12z} + \cdots ,\\
A_x(zM(z)^2,1,1) &= \frac 29 - \frac 29 \sqrt{1-12z} + \cdots ,\\
B^\bullet(zM(z)^2,1,1,1/M(z)) &=
{\frac { \left( 7-\sqrt {17}\right)  \left( 5-\sqrt {17} \right) }{
72}} - \frac{\left( 1+\sqrt {17} \right)  \left( -5+\sqrt {17} \right) ^{2} }{48} \sqrt{1-12z} + \cdots
\end{align*}
and so (\ref{eqLe41}) follows.

From (\ref{eqLe41}) it directly follows that
\[
[z^n]\, E_a(z) \sim \frac{5-\sqrt{17}}{2\sqrt{\pi}} n^{-3/2} 12^n.
\]
By dividing that by $M_n = [z^n] M(z) \sim (2/\sqrt{\pi}) n^{-5/2} 12^n$ the final result follows.
\end{proof}

\subsection{Second moment computations}\label{sec:comb-variance}

In principle the above combinatorial method can be extended to compute asymptotics
of higher moments. In what follows we indicate how this can be worked out for the 
second moment for the number of cut vertices in random planar maps. This leads, too, 
to an asymptotic expansion for the variance. As mentioned in the Introduction it is
expected that we have $\mathbb{V}{\rm ar}[X_n] \sim c_2 n$ for some constant $c_2 > 0$.

We set 
\[
E_a^{(2)}(z) = \left.\frac{\partial^2 M_a(z,y)}{\partial y^2} \right|_{y=1} \quad \mbox{and}\quad
E_r^{(2)}(z) = \left.\frac{\partial^2 M_r(z,y)}{\partial y^2 } \right|_{y=1}. 
\]
It is clear that 
\[
E_a^{(2)}(z) = \sum_{n\ge 0} M_n \mathbb{E}[X_n(X_n-1)] \, z^n.
\]
Since
\[
\mathbb{V}{\rm ar}[X_n]  = \mathbb{E}[X_n(X_n-1)] + \mathbb{E}[X_n] - \mathbb{E}[X_n]^2
\]
it is sufficient to obtain precise asymptotics for $\mathbb{E}[X_n(X_n-1)]$.

In principle we can work as above. However, instead of first derivatives we have to
consider second derivatives with respect of $y$ in the equations (\ref{1.6}) and (\ref{1.8-2aa}) 
(and a first derivative in equation (\ref{1.8-3})). From that it is possible to get an explicit 
expression for $E_a^{(2)}(z)$ in terms of the form
\[
\sum_{k, \ell \ge 1} \overline A_{w_kw_\ell} \left( z; M(z), M(z)^2, \ldots; 1 \right) k M(z)^{k+\ell-1} \quad\mbox{or}\quad
\sum_{k, \ell \ge 1} \overline A_{w_kw_\ell} \left( z; M(z), M(z)^2, \ldots; 1 \right)
\]
(and similar functions) and in terms of functions that have been already calculated.
All but one appearing functions can be handled easily by taking derivatives of (\ref{1.8}) or (\ref{eqsum3}).
What remains is to handle the double sum
\begin{equation}\label{eqdoublesum}
\sum_{k, \ell \ge 1} \overline A_{w_kw_\ell} \left( z; v, v^2, \ldots; 1 \right)
\end{equation}
which is the generating function of $2$-connected planar maps, where two faces are marked and discounted.

In order to handle (\ref{eqdoublesum}) we can proceed (again) as above but with slightly more care.
For the sake of brevity we do not work out the simple cases, where the root face is one of the marked faces.
The main problem is to handle those cases, where two different non-root faces are marked.
Let $B^{\bullet\bullet}(z,x,u,w_1,w_2)$ be the generating function of 2-connected planar maps, where
two different non-root faces, that are ordered, are marked and where $u$ takes care of
the root face valency, $w_1$ on the valency of the first marked face, and $w_2$ on the
valency of the second marked face.
By using (again) the fact, that a 2-connected planar map, where we delete the root edge, decomposes
into a sequence of 2-connected maps or single edges, we obtain the relation
\begin{align*}
B^{\bullet\bullet}(z,1,u,w_1,w_2)
&= zu w_1\frac { \frac{uB^\bullet(z,1,w_1,w_2)-w_1B^\bullet(z,1,u,w_2)}{w_1-u}}
{\left( 1- \frac{ uB(z,1,w_1)-w_1B(z,1,u)}{w_1-u} - zuw_1    \right)^2 } \\
&+ zu w_2\frac { \frac{uB^\bullet(z,1,w_2,w_1)-w_2B^\bullet(z,1,u,w_1)}{w_2-u}}
{\left( 1- \frac{ uB(z,1,w_2)-w_2B(z,1,u)}{w_2-u} - zuw_2    \right)^2 } \\
&+ zu \frac { \frac{uB^\bullet(z,1,1,w_1)-B^\bullet(z,1,u,w_1)}{1-u} \frac{uB^\bullet(z,1,1,w_2)-B^\bullet(z,1,u,w_2)}{1-u} }
{\left( 1- \frac{ uB(z,1,1)-B(z,1,u)}{1-u} - zu    \right)^3 } \\
&+ zu \frac { \frac{uB^{\bullet\bullet}(z,1,1,w_1,w_2)-B^{\bullet\bullet}(z,1,u,w_1,w_2)}{1-u}}
{\left( 1- \frac{ uB(z,1,1)-B(z,1,u)}{1-u} - zu    \right)^2 }.
\end{align*}
The first two terms on the right hand side correspond to the cases, where one of the marked faces is precisely
the second face of the root edge, the third term handles the case, where the two marked faces are in different
2-connected components of the above mentioned sequence, and finally the last term corresponds to the case,
where the two marked faces are in the same component. This is again a linear catalytic equation, that can be rewritten 
into the form 
\begin{align*}
& B^{\bullet\bullet}(z,1,u,w_1,w_2) \left( 1 + \frac{zu}{1-u} \frac 1{\left( 1- \frac{ uB(z,1,1)-B(z,1,u)}{1-u} - zu    \right)^2 } \right) \\
&= \frac{zu^2 B^{\bullet\bullet}(z,1,1,w_1,w_2)}{1-u}\frac 1{\left( 1- \frac{ uB(z,1,1)-B(z,1,u)}{1-u} - zu    \right)^2 } +
H(z,u,w_1,w_2),
\end{align*}
where $H(z,u,w_1,w_2)$ contains just already known functions. Again we set $u = u_1(z)$ so that the left hand side cancels and 
we obtain from the right hand side
\[
B^{\bullet\bullet}(z,1,1,w_1,w_2) = \frac{H(z,u_1(z),w_1,w_2)}{u_1(z)}.
\]
Finally we have to consider the specialization
\[
B^{\bullet\bullet}(zv^2,1,1,1/v,1/v) 
\]
to get the {\it main part} of (\ref{eqdoublesum}). (As mentioned above we skip the easier parts, where one of the
marked faces is the root face.)

This procedure leads to an explicit expression for $E_a^{(2)}(z)$. It is therefore clear that the
asymptotic analysis can be worked out. By the way all involved functions are algebraic which also shows
that $E_a^{(2)}(z)$ is algebraic, too.

\section{Central limit theorems 
}\label{sec:subcritical}

\subsection{Proof of Theorem~\ref{Th2}}
\begin{proof}
We consider a subcritical class $M^*$ of planar maps that satisfies the scheme (\ref{1.5-2}).
It is clear that we can apply the same procedure as in Section~\ref{sec:4.2} for taking cut vertices
into account. This leads to the system of equations (\ref{1.6}), \eqref{1.8-2aa}, (\ref{1.8-3})
for the unknown functions $M_0(x,y)$, $M_r(x,y)$, $M_a(x,y)$. 

If we set $y=1$ (that is, we just count planar maps in this subclass $M^*$ of size $n$)
then we have $M^*(z) = M_r(z,1) = M_a(z,1)$ and, thus, $M^*(z)$ satisfies the (single) equation
\[
M^*(z) = 1 + A(z M^*(z)^2, 1, 1).
\]
By assumption we are in a subcritical and aperiodic situation. By standard methods (see \cite[Theorem 2.19]{MR2484382})
it follows that $M^*(z)$ has a squareroot singularity at $z_1$ and a local representation of
the form
\[
M^*(z) = g(z) - h(z)\sqrt{1- z/z_1},
\]
where $g(z)$ and $h(z)$ are analytic at $z_1$ and $h(z_1) \ne 0$. Furthermore $z_1$ is the only
singularity on the circle $|z| = z_1$.

The main step for proving asymptotic normality is to show that we get the same singular behaviour
if $y$ varies around $y=1$ (see \cite[Theorem 2.21]{MR2484382}).
For this purpose it is sufficient to show that the function 
\begin{equation}\label{eqanalytic}
(z,y,M) \mapsto \overline A( z; yM -y+1, yM^2-y+1, \ldots; 1)
\end{equation}
is analytic for $|z| <  z_1 + \eta$, $|y-1| < \eta$, $|M| < M^*(z_1) + \eta$ for
some $\eta> 0$. 

Recall that 
\[
\overline A (z; v, v^2, \ldots; 1) = A(zv^2, 1,1)
\]
and that $A(zv^2, 1,1)$ is analytic as long $|zv^2| < z_0$.
We note that $M^*(z_1) > 1$ and recall that by assumption $z_1 M^*(z_1)^2 < z_0$. 
Hence there exists $\eta > 0$ with
\[
2\eta + \eta^2 \le 2 M^*(z_1) \eta \quad \mbox{and}\quad (z_1+\eta)M^*(z_1)^2 (1+3\eta)^2 < z_0.
\]
The first property is equivalent to 
\[
(1 + \eta)(M^*(z_1) + \eta) + \eta \le M^*(z_1) (1 + 3 \eta)
\]
which implies for all integers $r\ge 1$
\[
(1 + \eta)(M^*(z_1) + \eta)^r  + \eta \le (M^*(z_1) (1 + 3 \eta))^r.
\]
Now suppose that $|z| <  z_1 + \eta$, $|y-1| < \eta$, $|M| < M^*(z_1) + \eta$.
Then we have
\[
|yM^r - y + 1| \le (1+ \eta) (M^*(z_1) + \eta)^r  + \eta \le (M^*(z_1) (1 + 3 \eta))^r < z_0
\]
so that 
\begin{align*}
| \overline A( z; yM -y+1, yM^2-y+1, \ldots; 1) | &\le  \overline A( |z|; |yM -y+1|, |yM^2-y+1|, \ldots; 1)  \\
&\le \overline A( z_1 + \eta; M^*(z_1) (1 + 3 \eta), (M^*(z_1) (1 + 3 \eta))^2, \ldots; 1) \\
&=  A((z_1+\eta)M^*(z_1)^2 (1+3\eta)^2  , 1,1)
\end{align*}
Since $(z_1+\eta)M^*(z_1)^2 (1+3\eta)^2 < z_0$ this shows that the mapping (\ref{eqanalytic}) is
analytic in this range.

We can now do a similar computation for $\overline B (z; v^2, v^3, \ldots; 1)$. Summing up it follows that
the solution functions $M_0(z,y)$, $M_r(z,y)$, $M_a(z,y)$ system of equations (\ref{1.6}), \eqref{1.8-2aa}, (\ref{1.8-3}) 
have squareroot singularities of the form
\begin{align*}
M_0(z,y) &= g_0(z,y) - h_0(z,y) \sqrt{1 - z/\rho(y)}, \\
M_r(z,y) &= g_r(z,y) - h_r(z,y) \sqrt{1 - z/\rho(y)}, \\
M_a(z,y) &= g_a(z,y) - h_a(z,y) \sqrt{1 - z/\rho(y)}
\end{align*}
for $|y-1| < \eta$, where $\rho(y)$ is an analytic function with $\rho(1) = z_1$ and where
the functions $g_0,h_0,g_r,h_r,g_a,g_h$ are analytic for $|z-z_1| < \varepsilon$ and $|y-1|< \eta$
(for some $\varepsilon> 0$ and satisfy $h_0(z_1,1)\ne 0$, $h_r(z_1,1) \ne 0$, $h_a(z_1,1)\ne 0)$.
In principle we can apply \cite[Theorem 2.33]{MR2484382}, however, the positivity condition is 
only partly satisfied. Since we already know that we have squareroot singularities for $y=1$
we can easily circumvent this condition.

Since $z_1$ is the only singularity of the function $M^*(z) = M_a(z,1) = M_r(z,1)$ on the circle $|z| \le z_1$
it follows by continuity (and the implicit function theorem) that $\rho(y)$ is the only singularity 
of the functions $z \mapsto M_0(z,y)$, $z \mapsto M_r(z,y)$, $z \mapsto M_a(z,y)$ if $y$ is sufficiently
close to $1$.  

This completes the proof of a central limit theorem of the form (\ref{eqTh2}), compare with \cite[Theorem 2.23]{MR2484382}.
\end{proof}

\subsection{Proof of Theorem \ref{th:blocks}}

\begin{proof}

The proof is based on the univariate version Equation \eqref{1.5-2}, namely 
$$
M(z) = 1+A(zM(z)^2).
$$
We add a new variable marking non-root blocks and we get 
$$
M(z,w) = 1+A(z(1+w(M(z,w)-1))^2).
$$
We know that the composition scheme is critical for $w=1$, hence by continuity it is also critical (with a singularity of type $3/2$) for $w$ sufficiently close to $1$. . This implies that the singularities of $M(z,w)$ come from those of $A$. The unique dominant singularity of $A(z)$ is at $z=4/27$, hence, for fixed $w$ near 1, $M(z,w)$ has a singularity at $\rho(w)$ given by
$$
\frac{4}{27} = \rho(w) (1+w(M(\rho(w),w)-1))^2.
$$
The fastest way to compute $\rho(w)$ is to compute the minimal polynomial of $M(z,u)$ by elimination using the minimal polynomial of $A(z)$. We obtain a polynomial $P(z,w,M(z,w))$ of degree 4 in $M$, whose discriminant is equal to 
$$
D(z) = z \left( w-1 \right) ^{2} \left( 256\,{w}^{3}{z}^{2}-32\,{w}^{2}z+1
\right)  \left( 3\,{w}^{2}z+18\,wz+27\,z-4 \right) ^{3}.$$
By general principles \cite{MR2483235}, the singularities of $M$ must be among the roots of the discriminant. We discard the first two trivial factors and we are left with two candidates. Setting $w=1$ we must recover the singularity $z=1/12$ of the univariate function $M(z)$, and this implies that the right factor is the last one. Solving for $z$ we get the unique singularity  $\rho(w)$ as 
$$
	\rho(w) = \frac{4}{3(w^2+6w+9)}.
	$$
This implies a central limit theorem for the associated random variable as in the previous section with moments
$$
	\mathbb{E}(X_n) \sim -\frac{\rho'(1)}{\rho(1)}n = \frac{n}{2}, \qquad \mathbb{V}\hbox{ar}(X_n) \sim  -\frac{\rho''(1)}{\rho(1)}n + \mathbb{E}(X_n)+ \mathbb{E}(X_n)^2 = \frac{3n}{8}.$$
 This concludes the proof. 
\end{proof}

%
%
%

\subsection{Outerplanar maps}

We want to illustrate that Theorem~\ref{Th2} can be extended to some further subclasses of planar maps 
like outerplanar maps. 
For this case we will prove the central limit theorem for the number of cut vertices in two different ways, 
first with the help of generating function and second with probabilistic arguments.

\subsubsection{Outerplanar maps with $n$ vertices}

We recall that the generating function $M_O(z)$ of outerplanar maps satisfies (\ref{eqGFouterplanar}), 
where the function 
\[
A_O(z) =  \frac 14 \left( 1+ z - \sqrt{1 -6z + z^2} \right)
\]
is the generating function for polygon dissections (plus a single edge) 
has radius of convergence  $z_{0,O} = 3 - 2\sqrt 2$. 
From this we obtain
\[
M_O(z) = \frac{z \left( 3 - \sqrt{1-8z} \right)}{2(1+z)}.
\]
The radius of convergence of $M_O(z)$ is $z_{1,O} = \frac 18$ so that $M_O(z_{1,O}) = \frac 1{18}
< z_{0,O}$. Note that  $M_O(z)$ has a squareroot singularity (as it has to be). 
Now let $M_O(z,y)$ denote the generating function of outerplanar maps, where $y$ takes care of the
number of cut-vertices. We already mentioned that $M_O(z,y)$ satisfies the functional equation
\[
M_O(z,y) = \frac z{1- A_O(z + y(M_O(z,y)-z))}
\]
which gives
\[
M_O(z,y) = \frac{z\left( 3-z+yz - \sqrt{(y-1)z^2 - (6+2y)z +1 }  \right)}{2(1+yz)}.
\]
Clearly,  if $y$ is sufficiently close to $1$ then the singularities
of $M_O(z,y)$ and $A_O(z)$ do not interact and so we obtain a squareroot singularity 
\[
\rho(y) = \frac{3 + y - 2 \sqrt{2 + 2y}}{(y-1)^2}.
\]
for the mapping $z\mapsto M_O(z,y)$. Note that $\rho(y)$ is actually regular at $y=1$ and satisfies
$\rho(1) = 1/8$.

By \cite[Theorem 2.25]{MR2484382} we immediately obtain a central limit theorem with $\mathbb{E}[X_n] = c\, n + O(1)$ and
variance $\mathbb{V}{\rm ar}[X_n] = \sigma^2 n + O(1)$, where
\[
c = - \frac{\rho'(1)}{\rho(1)} = \frac 14 \quad\mbox{and}\quad
\sigma^2 = - \frac{\rho''(1)}{\rho(1)} + \mu + \mu^2 = \frac 5{32}.
\]

Next we show how this central limit theorem can be obtained by probabilistic tools.
As illustrated in Figure~\ref{fi:decompnew2}, any outerplanar map $O$ with $n$ vertices corresponds bijectively to a planted plane tree $T(O)$ with $n$ vertices and a family $(\beta(v))_{v \in T(O)}$ of ordered sequences of dissections of polygons such that the  the outdegree of a vertex $v \in T(O)$ agrees with the number of non-root vertices in the sequence $\beta(v)$. Details on this decomposition may be found in~\cite[Sec. 2]{aihpstufler2017}.

\begin{figure}[h]
	\centering
	\begin{minipage}{1.0\textwidth}
		\centering
	\includegraphics[width=1.0\textwidth]{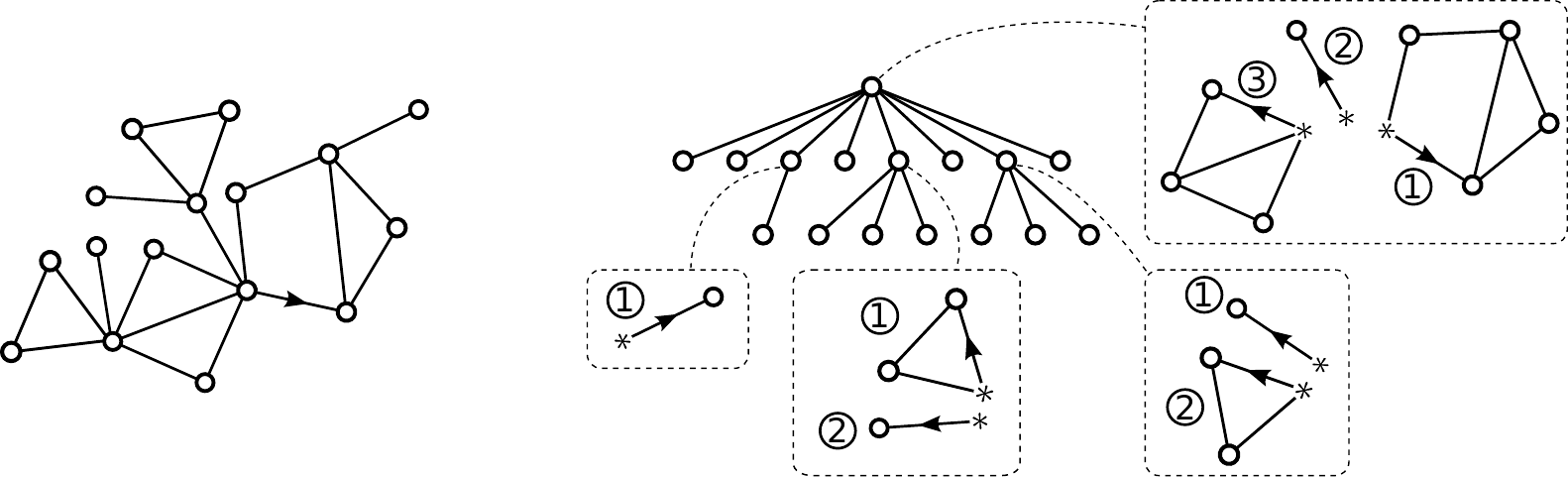}
		\caption{The decomposition of simple outerplanar rooted maps into decorated trees.\protect\footnotemark}
		\label{fi:decompnew2}
	\end{minipage}
\end{figure}

\footnotetext{Source of image: \cite[Fig. 2]{aihpstufler2017}.}

The root-vertex of $O$ corresponds to the root-vertex of $T(O)$. Any non-root vertex in $O$ is a cut-vertex if and only if it is not a leaf of $T(O)$. That is, the number $\mathrm{Cut}(O)$ of cut vertices in $O$ and the number $\mathrm{L}(T(O))$ of leaves in $T(O)$ are related by

\begin{align}
\label{eq:asdf}
\mathrm{Cut}(O) = (n-1) - \mathrm{L}(T(O)) + 
 \mathbf{1}_{\text{root of $O$ is a cut vertex}}.
\end{align}

If $\mO_n$ is the uniform outerplanar map with $n$ vertices, then $\cT_n := T(\mO_n)$ is a simply generated tree, obtained from conditioning a critical Galton--Watson tree on having $n$ vertices. The fact that outerplanar maps are subcritical in the sense of~\eqref{eq:subcout} ensures that the offspring distribution $\xi$  of the Galton--Watson tree may be chosen to satisfy $\Ex{\xi}=1$ and have finite exponential moments. By standard branching processes results (see for example~\cite{zbMATH06538878}) it holds that the number of leaves of $\cT_n$ satisfies a normal central limit theorem 

\begin{align}
\frac{\mathrm{L}(\cT_n) - n p_0}{\sqrt{n}} \convd N(0, \gamma^2),
\end{align}
with 
\begin{align}
p_0 := \Pr{\xi=0} \qquad \text{and} \qquad \gamma^2 := p_0 - p_0^2(1 + 1/\Va{\xi}).
\end{align}
By Equation~\eqref{eq:asdf} it follows that
\begin{align}
\label{eq:normal}
\frac{\mathrm{Cut}(\mO_n) - n (1- p_0)}{\sqrt{n}} \convd N(0, \gamma^2).
\end{align}
Equation~\eqref{eqGFouterplanar-2} enables us to determine the offspring distribution $\xi$ explicitly (see \cite[Sec. 4.2.1]{aihpstufler2017}), and show that 
\[
\Ex{\xi}=1, \qquad \Va{\xi} = 18, \qquad \Pr{\xi=0} = 3/4.
\]
Thus
\begin{align}
\frac{\mathrm{Cut}(\mO_n) - n/4}{\sqrt{n}} \convd N(0, 5/32).
\end{align}

\subsubsection{Bipartite outerplanar maps with $n$ vertices}

Finally we discuss bipartite outerplanar maps. 
Here we have again the relation (\ref{eqGFouterplanar}), however, the generating function $A_O(z)$ has to
be replaced by the generating function $A_O^{\mathrm{bip}}(z)$ of bipartite polygon dissections (plus a single edge). As illustrated in Figure~\ref{fi:diss}, any dissection may be decomposed into a root-edge and a series composition of other dissections.

\begin{figure}[h]
	\centering
	\begin{minipage}{0.8\textwidth}
		\centering
	\includegraphics[width=0.6\textwidth]{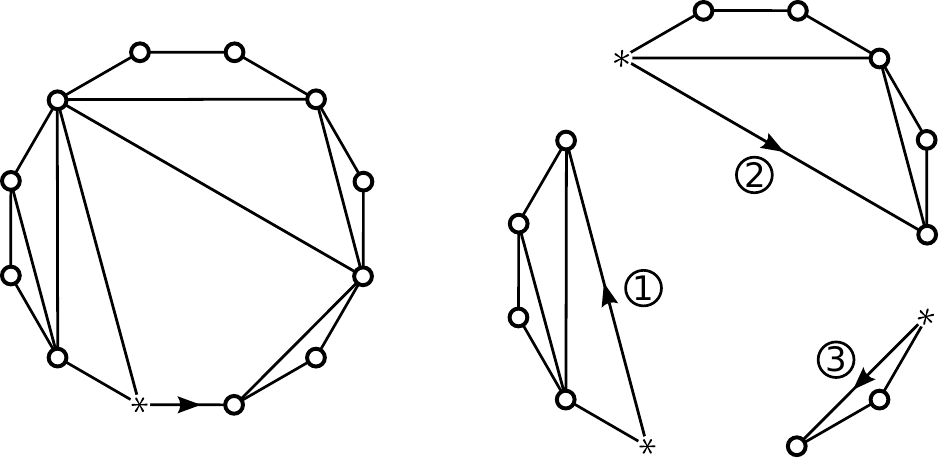}
		\caption{The decomposition of edge-rooted dissections of polygons.\protect\footnotemark}
		\label{fi:diss}
	\end{minipage}
\end{figure}  
\footnotetext{Source of image: \cite[Fig. 4]{aihpstufler2017}.}
Such a dissection is bipartite, if and only if all of its parts are bipartite and the number of parts is odd. Hence $A_O^{\mathrm{bip}}(z)$ is the solution of the equation
\begin{align}
A_O^{\mathrm{bip}}(z) = z + \frac{A_O^{\mathrm{bip}}(z)^3}{1-A_O^{\mathrm{bip}}(z)^2}.
\end{align}
The radius of convergence of
$A_O^{\mathrm{bip}}(z)$ equals \[
z_{0,O} = -\frac{1}{8} \sqrt{5-\sqrt{17}} \left(\sqrt{17}-7\right) = 0.33674\ldots.
\] The radius of convergence of the series $M_O^{\mathrm{bip}}(z)$ (satisfying
$M_O^{\mathrm{bip}}(z) = z/(1-A_O^{\mathrm{bip}}(M_O^{\mathrm{bip}}(z)))$ is then given by $z_{1,O} = -5 + 3 \sqrt{3} = 0.19615\ldots$ and we have
\[
M_O^{\mathrm{bip}}(z_{1,O}) = \frac{2}{3} \left(2 \sqrt{3}-3\right) = 0.309401\ldots < z_{0,O} .
\]
 Consequently we are (again)
in a subcritical situation and obtain (as above) a central limit theorem. By a more refined 
analysis we also obtain
\[
c = \frac{\sqrt{3}-1}{2} \quad \mbox{and} \quad \sigma^2 = \frac{11 \sqrt{3}-17}{12}.
\]

As in the case of (all) outerplanar maps it is also possible to prove the central 
limit theorem by probabilistic tools.
Note that an outerplanar map is bipartite if and only if all its blocks are. Hence the bijection in Figure~\ref{fi:decompnew2} restricts to a bijection between bipartite outerplanar maps and plane trees decorated by ordered sequences of bipartite dissections. In particular, the uniform random bipartite planar map $\mO_n^\mathrm{bip}$ may be generated by decorating a simply generated tree $\cT_n^\mathrm{bip}$, obtained by conditioning some $\xi^\mathrm{bip}$-Galton--Watson tree.
 This allows us to explicitly determine the offspring distribution $\xi^\mathrm{bip}$, yielding  (see \cite[Sec. 4.2.2]{aihpstufler2017})
\[
\Ex{\xi^\mathrm{bip}}=1, \qquad \Va{\xi^\mathrm{bip}} = 9(\sqrt{3}-1), \qquad \Pr{\xi^\mathrm{bip}=0} = (3 - \sqrt{3})/2.
\]
Equation~\ref{eq:normal} holds analogously for $\mO_n^{\mathrm{bip}}$ and $\xi^\mathrm{bip}$, yielding
\begin{align}
\frac{\mathrm{Cut}(\mO_n^{\mathrm{bip}}) - n(-1 + \sqrt{3})/2}{\sqrt{n}} \convd N(0, (-17 + 11\sqrt{3})/12).
\end{align}

\bibliographystyle{siam}
\bibliography{cutvertices}

\end{document}